\documentclass [a4paper,10pt,article,oneside]{memoir}

\usepackage{amsmath}
\usepackage{amsfonts,mathrsfs,paralist,amssymb,bm,amsthm}
\usepackage{graphicx,color,paralist}
\usepackage{url}
\usepackage{mathrsfs}
\usepackage[all]{xy}
\xyoption{poly}

\usepackage{biblatex}
\bibliography{mobius.bib}

\setsecnumformat{\csname the#1\endcsname---}
\setsecheadstyle{\centering\Large\bfseries\sffamily}

\setsecnumdepth{subsection}

\everymath{\normalfont}

\newtheoremstyle{thm}
     {.8ex plus .3ex minus .1ex}
     {.8ex plus .3ex minus .1ex}
     {\itshape}
     {}
     {\bfseries}
     {---}
     {0em}
     {$\bullet$\hbox{\ }#1\hbox{\ }#2}

\theoremstyle{thm}

\newtheorem{definition}{Definition}[section]
\newtheorem{theorem}[definition]{Theorem}
\newtheorem{lemma}[definition]{Lemma}
\newtheorem{proposition}[definition]{Proposition}
\newtheorem{corollary}[definition]{Corollary}

\newcommand{\etat}[1]{*+<1em,1em>[o][F-]{\makebox[1.5em]{\strut $#1$}}}
\renewcommand{\1}[1]{\mathbf{1}_{\{#1\}}}
\newcommand{\tq}{\;:\;}
\newcommand{\iid}{{\normalfont\slshape\textsf{\small i.i.d.}}}
\newcommand{\lub}{{\normalfont\slshape\textsf{\small l.u.b.}}}

\newcommand{\indep}{\parallel}
\newcommand{\up}[1]{\,\uparrow\! #1}
\newcommand{\seq}[2]{(#1_{#2})_{#2\geq0}}
\newcommand{\NN}{\mathbb{N}}
\newcommand{\ZZ}{\mathbb{Z}}
\newcommand{\pr}{\mathbb{P}}
\newcommand{\esp}{\mathbb{E}}
\newcommand{\prq}{\mathbb{Q}}
\newcommand{\A}{\mathcal{A}}
\newcommand{\C}{\mathscr{C}}
\newcommand{\Cstar}{\mathfrak{C}}
\newcommand{\RR}{\mathbb{R}}
\newcommand{\D}{\mathcal{D}}
\newcommand{\FFF}{\mathfrak{F}}
\newcommand{\GGG}{\mathfrak{G}}
\newcommand{\slgb}{\mbox{$\sigma$-al}\-ge\-bra}
\newcommand{\pas}{\text{$\pr$-a.s.}}
\newcommand{\M}{\mathcal{M}}
\newcommand{\muc}{\mu_{\M_c}}
\newcommand{\ECF}{\mbox{\normalfont\sffamily xCF}}
\newcommand{\R}{\mathscr{R}}
\renewcommand{\H}{\mathcal{H}}
\newcommand{\W}{\mathcal{W}}
\newcommand{\rest}[1]{\bigl|_{#1}}
\newcommand{\height}{\tau}
\newcommand{\cut}{\kappa}
\newcommand{\pirreducible}{\mbox{irre}\-\mbox{ducible}}
\newcommand{\B}{\partial}
\newcommand{\MM}{\mathbf{M}}
\newcommand{\speed}{\gamma}
\newcommand{\rev}{\mathsf{rev}}
\newcommand{\wbar}{\widehat{w}}
\newcommand{\goodname}{M\"{o}bius}
\newcommand{\Goodname}{M\"{o}bius}
\newcommand{\un}{\bm{1}}

\begin{document}

\mainmatter

\begin{center}
\huge\bfseries\sffamily
Uniform and Bernoulli measures\\
on the boundary of trace monoids
\end{center}

\medskip

\begin{tabular}{cc}
  \Large\sffamily Samy Abbes&\Large\sffamily Jean Mairesse\\
  \small University Paris Diderot -- Paris~7&\small CNRS, LIP6 UMR 7606\\
  \small CNRS Laboratory PPS (UMR 7126)&\small Sorbonne Universit\'es\\
  \small Paris, France&\small  UPMC Univ Paris 06, France\\
  \small
  \ttfamily\footnotesize  samy.abbes@univ-paris-diderot.fr
&\ttfamily\footnotesize jean.mairesse@lip6.fr
\end{tabular}

\begin{abstract}
  \vspace{-1em} Trace monoids and heaps of pieces appear in various
  contexts in combinatorics. They also constitute a model used in
  computer science to describe the executions of asynchronous
  systems. The design of a na\-tu\-ral probabilistic layer on top of
  the model has been a long standing challenge. The difficulty comes
  from the presence of commuting pieces and from the absence of a
  global clock.  In this paper, we introduce and study the class of
  \emph{Bernoulli} probability measures that we claim to be the
  simplest adequate probability measures on infinite traces. For this,
  we strongly rely on the theory of trace combinatorics with the
  M\"{o}bius polynomial in the key role. These new measures provide a
  theoretical foundation for the probabilistic study of concurrent
  systems.
\end{abstract}

\section{Introduction}
\label{sec:introduction-1}

Trace monoids are finitely presented monoids with commutation
relations between some generators, that is to say, relations of the
form $a\cdot b= b \cdot a$.

Trace monoids have first been studied in combinatorics under the name
of partially commutative monoids~\cite{cartier69}. It was noticed by
Viennot that trace monoids are ubiquitous both in combinatorics and in
statistical physics~\cite{viennot86}. Trace monoids have also
attracted a lasting interest in the computer science community, since
it was realized that they constitute a model of concurrent systems
which are computational systems featuring parallel actions; typically,
parallel access to distributed databases or parallel events in
networked systems~\cite{diekert90,diekert95}. In a nutshell, the
co-occurrence of parallel actions corresponds to the commutation
between generators in the trace monoid.  The relationship with other
models of concurrency has been extensively
studied~\cite{rozoy91,winskel95}. In particular, in most concurrency
models, the executions can be described as regular trace languages,
that is to say, regular subsets of trace monoids. Hence trace monoids
are among the most fundamental objects of concurrency theory.

\smallskip

There are several motivations for adding a probabilistic layer on top
of trace monoids.  In the concurrent systems context, it is relevant
for network dimensioning and performance
evaluation~\cite{saheb89,krob03}.  It is also a question that has been
considered for general combinatorial structures since the 80's, and
which is crucial for the design of random sampling
algorithms~\cite{jerrum86}. Consider for instance the model checking
of asynchronous systems.  Such systems are known to suffer from the
``state-space explosion'' problem.  So it is in practice impossible to
check for all the trajectories. The key idea in \emph{statistical
  model checking} is to design testing procedures, relying on random
sampling, that provide quantitative guarantees for the fair
exploration of trajectories.  In this paper, we design a relevant
probabilistic layer at the level of the full trace monoid. This is a
first and necessary step, which has to be thoroughly understood,
before pushing the analysis further towards the regular trace
languages describing the trajectories of concurrent systems.

\medskip

The elements of a trace monoid are called traces. Traces can be seen
as an extended notion of words, where some letters are allowed to
commute with each other. Traces carry several notions which are
transposed from words. In particular, traces have a natural notion of
length, the number of letters in any representative word, and they are
partially ordered by the prefix relation inherited from words.  A
trace monoid can be embedded into a compact metric space where the
boundary elements are \emph{infinite traces}, which play the same role
with respect to traces than infinite words play with respect to words.

\medskip

One of our goals is to design a natural and effective notion of
``uniform'' probability measure on \emph{infinite traces}.  Let us
illustrate the difficulties that have to be overcome.

\medskip

\noindent\textbf{An elementary challenge with no elementary solution.} 
Consider the basic trace monoid $\M=\langle a,b,c\;|\; a\cdot b=b\cdot
a\rangle$. For $u\in \M$, denote by $\up u$ the set of all infinite
traces for which $u$ is a possible prefix. Does there exist a
probability $\pr$ on infinite traces which is \emph{uniform}, that is,
which satisfies: $\pr(\up u)=\pr(\up v)$ if $u$ and $v$ have the same
length~?

A first attempt consists in performing a random walk on the trace
monoid (see \cite{VNBi}). Draw a random sequence of letters in
$\{a,b,c\}$, the letters being chosen independently and with
probabilities $p_a,p_b,p_c$. Then consider the infinite trace obtained
by concatenating the letters.  We invite the reader to check by hand
that the probability measure thus induced on infinite traces does not
have the uniform property, whatever the choice of $p_a,p_b,p_c$\,.

A second attempt consists in considering the well-defined sequence
$(\nu_n)_{n\geq0}$, where $\nu_n$ is the uniform measure on the finite
set $\M_n=\{u\in\M\tq |u|=n\}$ of traces of length~$n$. Observe that
there are three traces of length~$1$, which are~$a$, $b$~and~$c$, and
eight traces of length~$2$, obtained from the collection of nine words
of length $2$ on $\{a,b,c\}$ by identifying the two words $ab$
and~$ba$. In particular, we have:
\begin{gather*}
1/3 =\nu_1(a)\;\neq\;\nu_2(aa)+\nu_2(ab)+\nu_2(ac)=3/8\,.
\end{gather*}
Hence the pair $(\nu_1,\nu_2)$ is \emph{not consistent}; and neither
is the pair $(\nu_n,\nu_{n+1})$ for all $n>0$.  There is therefore no
probability measure on infinite traces that induces the family
$(\nu_n)_{n\geq0}$\,.

We seemingly face the Cornelian
dilemma of choosing between consistent but non-uniform probabilities (first attempt),
or uniform but non consistent probabilities (second attempt).  

\medskip

\noindent\textbf{Results.} In this paper, we prove the existence and uniqueness of
the uniform probability measure $\pr$ for any irreducible trace
monoid---irreducibility corresponds to a connectedness property.  The
above dilemma is solved by playing on a variable which was thought to
be fixed: the total mass of finite marginals.  Indeed, the uniform
\emph{probability} on infinite traces induces a uniform \emph{measure}
on traces of a given length \emph{whose mass exceeds~$1$}.  As for the
consistency conditions, they do not hold and they are replaced by
compatibility conditions based on the inclusion-exclusion principle.

The uniform measure has the remarkable property of satisfying
$\pr\bigl(\up(u\cdot v)\bigr)=\pr(\up u)\pr(\up v)$ for any traces
$u,v\in\M$. More generally, we call Bernoulli measure any probability
measure satisfying this identity, which corresponds to a memoryless
property on traces. We exhibit infinitely many Bernoulli measures and
characterize all of them by means of a finite family of
intrinsic parameters obeying polynomial equations.  Furthermore,
we establish a realization theorem by proving that Bernoulli measures
correspond to some particular Markov chains on the Cartier-Foata
decomposition of traces. This
realization result is a basis for further work on algorithmic
sampling, a task that has not been tackled in the
literature so far. 

The M\"obius polynomial associated with the trace monoid appears in
all the results.  For instance, we establish that the uniform measure
satisfies $\pr(\up u)=p_0^{|u|}$, for all $u\in \M$, where $p_0$ is
the unique root of smallest modulus of the M\"obius polynomial. Also,
in the realization result for Bernoulli measures, the relationship
between the intrinsic parameters and the transition matrix of the
Markov chain is based on a general M\"obius transform in the sense of
Rota~\cite{rota64}. This highlights the deep combinatorial structure
of the probabilistic objects that we construct.

\medskip

\noindent\textbf{Related work.}
The uniform measure that we construct is closely related to two
classical objects: the Parry measure and the Patterson-Sullivan
measure. The \emph{Parry measure} is the measure of maximal entropy on
a sofic subshift, that is, roughly, the ``uniform'' measure on the
infinite paths in a finite automaton~\cite{lind95}. Traces can be
represented by their Cartier-Foata decompositions which are recognized
by a finite automaton having an associated Parry measure.  The
limitation in this approach is that the link with the combinatorics of
the trace monoid remains hidden in the construction. In a sense, our
results reveal the inherent combinatorial structure of the Parry
measure.  The \emph{Patterson-Sullivan measure} is also a uniform
measure, which is classically constructed on the boundary at infinity
of some geometric groups~\cite{kaimanovich90}. The proof of its
existence is non-constructive and is based on the Poincar\'e series of
the group, which, in the context of the trace monoid, is simply
$\sum_{u\in \M}z^{|u|}$\,. Using that the Poincar\'e series of $\M$ is
the inverse of the M\"obius polynomial, we get an explicit and
combinatorial identification of the Patterson-Sullivan measure for
trace monoids. Hence our results provide the first discrete framework,
outside trees~\cite{coornaert93}, where the Patterson-Sullivan measure
is explicitly identified.

Our approach radically differs from the probabilistic techniques found
in the computer science literature and related to concurrent systems,
such as Rabin's probabilistic automata~\cite{rabi63} and their
variants, probabilistic process algebra~\cite{hill96}, or stochastic
Petri nets~\cite{haas}.  All these approaches rely first on a
transposition of the asynchronous system into a sequential one, after
which a Markov chain structure is typically added. In contrast, we
consider the randomization of the elements involving parallelism, and
not sequentializations of those elements.

\medskip

\noindent\textbf{Organization of the paper.}
The paper is organized in four sections.
Section~\ref{part:framework-results} exposes the framework and
contains the statements of the results, with no proofs.
Section~\ref{sec:exampl-appl} illustrates the results through a study
of two concrete examples. It also provides a first immediate
application of our constructions to the computation of the ``speedup''
of trace monoids.  Section~\ref{part:auxiliary-tools} introduces
auxiliary tools.  Section~\ref{part:proofs-results} is devoted to the
proofs of the results stated in Section~\ref{part:framework-results}.
Last, a concluding section provides perspectives for future work.

\section{Framework and results}
\label{part:framework-results}

\subsection{Trace monoids and their boundary}
\label{sec:trace-monoids-their}

An \emph{independence pair} is an ordered pair $(\Sigma,I)$, where
$\Sigma$ is a finite set, referred to as the \emph{alphabet} and whose
elements are called \emph{letters}, and $I\subset \Sigma\times \Sigma$
is an irreflexive and symmetric binary relation on~$\Sigma$.  To each
independence pair is attached another ordered pair $(\Sigma,D)$,
called the associated \emph{dependence pair}, where $D$ is defined by
$D=(\Sigma\times\Sigma)\setminus I$, which is a symmetric and
reflexive relation on~$\Sigma$. Two letters $\alpha,\beta\in\Sigma$
such that $(\alpha,\beta)\in I$ are said to be \emph{parallel},
denoted by $\alpha\indep \beta$\,.

To each independence pair $(\Sigma,I)$ is associated the finitely
presented monoid
\begin{equation*}
\M(\Sigma,I)=\bigl\langle\Sigma\,|\,
\alpha\cdot\beta=\beta\cdot\alpha\text{ for $(\alpha,\beta)\in
  I$}\bigr\rangle\,.
\end{equation*}
Denoting by $\Sigma^*$ the free monoid generated by~$\Sigma$, the
monoid $\M=\M(\Sigma,I)$ is thus the quotient monoid~$\Sigma^*/\R$\,,
where $\R$ is the congruence relation on $\Sigma^*$ generated by
$(\alpha\beta,\beta\alpha)$, for $(\alpha,\beta)$ ranging over~$I$.
Such a monoid $\M$ is called a \emph{trace monoid}, and its elements
are called \emph{traces}. The concatenation in $\M$ is denoted by the
dot~``\,$\cdot$\,''\,, the unit element in $\M$, the {\em empty}
trace, is denoted~$0$.  The trace monoid $\M$ is said to be
\emph{non-trivial} if $\Sigma\neq\emptyset$. \emph{By convention, we
  only consider non-trivial trace monoids throughout the paper, even
  if not specified.}  Throughout the paper, we consider a generic
trace monoid $\M=\M(\Sigma,I)$.

\medskip

Viennot's \emph{heap of pieces} interpretation is an enlightening
visualization of traces~\cite{viennot86}. In this interpretation, a
trace is identified with the \emph{heap} obtained from any
representative word as follows: each letter corresponds to a piece
that falls vertically until it is blocked; a letter is blocked by all
other letters but the ones which are parallel to it. We illustrate
this in Figure~\ref{fig:heapm1} for the example monoid $\M_1=\langle
a,b,c\,|\,a\cdot b=b\cdot a\rangle$\,.

\medskip

\begin{figure}[hb]
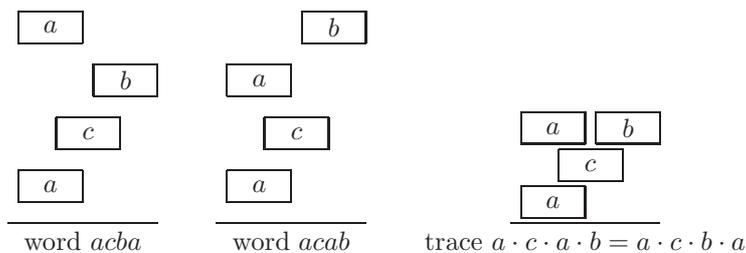

  \centering
  \begin{tabular}{ccc}
\xy
<.1em,0em>:
(0,6)="G",
"G"+(12,6)*{a},
"G";"G"+(24,0)**@{-};"G"+(24,12)**@{-};"G"+(0,12)**@{-};"G"**@{-},
(14,26)="G",
"G"+(12,6)*{c},
"G";"G"+(24,0)**@{-};"G"+(24,12)**@{-};"G"+(0,12)**@{-};"G"**@{-},
(0,66)="G",
"G"+(12,6)*{a},
"G";"G"+(24,0)**@{-};"G"+(24,12)**@{-};"G"+(0,12)**@{-};"G"**@{-},
(28,46)="G",
"G"+(12,6)*{b},
"G";"G"+(24,0)**@{-};"G"+(24,12)**@{-};"G"+(0,12)**@{-};"G"**@{-},
(-4,-2);(52,-2)**@{-}
\endxy& \quad
\xy
<.1em,0em>:
(0,6)="G",
"G"+(12,6)*{a},
"G";"G"+(24,0)**@{-};"G"+(24,12)**@{-};"G"+(0,12)**@{-};"G"**@{-},
(14,26)="G",
"G"+(12,6)*{c},
"G";"G"+(24,0)**@{-};"G"+(24,12)**@{-};"G"+(0,12)**@{-};"G"**@{-},
(0,46)="G",
"G"+(12,6)*{a},
"G";"G"+(24,0)**@{-};"G"+(24,12)**@{-};"G"+(0,12)**@{-};"G"**@{-},
(28,66)="G",
"G"+(12,6)*{b},
"G";"G"+(24,0)**@{-};"G"+(24,12)**@{-};"G"+(0,12)**@{-};"G"**@{-},
(-4,-2);(52,-2)**@{-}
\endxy
&\quad \xy
<.1em,0em>:
0="G",
"G"+(12,6)*{a},
"G";"G"+(24,0)**@{-};"G"+(24,12)**@{-};"G"+(0,12)**@{-};"G"**@{-},
(14,14)="G",
"G"+(12,6)*{c},
"G";"G"+(24,0)**@{-};"G"+(24,12)**@{-};"G"+(0,12)**@{-};"G"**@{-},
(0,28)="G",
"G"+(12,6)*{a},
"G";"G"+(24,0)**@{-};"G"+(24,12)**@{-};"G"+(0,12)**@{-};"G"**@{-},
(28,28)="G",
"G"+(12,6)*{b},
"G";"G"+(24,0)**@{-};"G"+(24,12)**@{-};"G"+(0,12)**@{-};"G"**@{-},
(-4,-2);(52,-2)**@{-}
\endxy\\
word $acba$&\quad word $acab$&\quad trace $a\cdot c\cdot a\cdot b=a\cdot c\cdot b\cdot a$
\end{tabular}
\caption{\textsl{Two congruent words
    and the resulting heap (trace) for~$\M_1$}}
  \label{fig:heapm1}
\end{figure}

The trace monoid $\M$ is said to be \emph{irreducible} whenever the
associated dependence pair~$(\Sigma,D)$, as an undirected graph, is
connected.  Note that a given trace monoid $\M(\Sigma,I)$ determines
the independence pair $(\Sigma,I)$ up to isomorphism, and hence the
dependence pair $(\Sigma,D)$ as well, which makes the definition of an
irreducible trace monoid intrinsic to the monoid.

\medskip

The notion of independence clique is central in the combinatorics of
trace monoids. An element $c\in\M$ is said to be a \emph{clique} if it
is of the form $c=\alpha_1\cdot\ldots\cdot\alpha_n$ for some integer
$n\geq0$ and for some letters $\alpha_1,\ldots,\alpha_n\in\Sigma$ such
that $i\neq j\implies \alpha_i\indep \alpha_j$\,.  The set of
cliques is denoted~$\C_\M$\,, or simply~$\C$. The set
$\C\setminus\{0\}$ of non-empty cliques is denoted\/~$\Cstar_\M$\,,
or\/~$\Cstar$\,.  Noting that a letter may occur at most once in any
representative word of a clique, we identify a clique with the set of
letters occurring in any of its representative words. In the heap
representation, each layer of a heap is a clique.

Two cliques $c,c'\in\C$ are said to be \emph{parallel} whenever
$c\times c'\subseteq I$\,, which is denoted $c\parallel c'$\,. This
relation extends to cliques the parallelism relation defined on
letters.  Observe that, since $I$ is supposed to be irreflexive, two
parallel cliques are necessarily disjoint.

For each clique $c\in\C$, we consider the sub-monoid $\M_c\subseteq\M$
defined as follows:
\begin{align}
  \label{eq:39}
  \Sigma_c&=\{\alpha\in\Sigma\tq\alpha\indep c\}\,,&
I_c&=I\cap(\Sigma_c\times\Sigma_c)\,,&
\M_c&=\M(\Sigma_c,I_c)\,.
\end{align}

So for instance, $\Sigma_0=\Sigma$ and $\M_0=\M$; while $\Sigma_c=\emptyset$ if $c$ is
a maximal clique, and then $\M_c=\{0\}$\,.

\medskip

The \emph{length} of a trace $u\in\M$ is defined as the length of any
of its representative words in the free monoid~$\Sigma^*$, and is
denoted by~$|u|$. Obviously, length is additive on traces, and $0$ is
the unique trace of length~$0$. The length of a trace corresponds to
the number of pieces in the associated heap. 

\medskip
We consider the left divisibility relation of~$\M$, denoted~$\leq$,
and defined by:
\begin{equation*}
  \forall u,v\in\M\quad u\leq v\iff\exists w\in\M\quad v=u\cdot w\,.
\end{equation*}
Trace monoids are cancellative~\cite{cartier69}. This justifies the
notation $v-u$ to denote the unique trace $w\in\M$ such that $v=u\cdot
w$ whenever $u\leq v$ holds. The two properties mentioned above for
the length of traces imply that $(\M,\leq)$ is a partial order.

\medskip
Informally, infinite traces correspond to heaps with a
countably infinite number of pieces. 
Following~\cite{abbes08}, a formal way to define infinite traces
associated to $\M$ is to consider the completion of $\M$ with respect
to least upper bound (\lub) of non-decreasing sequences in
$(\M,\leq)$. Say that a sequence $\seq uk$ is \emph{non-decreasing} in
$\M$ if $u_k\leq u_{k+1}$ holds for all integers $k\geq0$.  Let
$(\H,\preccurlyeq)$ be the pre-ordered set of all non-decreasing sequences
in~$\M$ equipped with the Egli-Milner pre-ordering relation, defined
as follows:
\begin{equation*}
\seq uk\preccurlyeq \seq{u'}k  \iff\forall k\geq0\quad\exists
k'\geq0\quad u_k\leq u'_{k'}\,.
\end{equation*}

Finally, let $(\W,\leq)$ be the collapse partial order associated with
$(\H,\preccurlyeq)$. The elements of $\W=\W(\Sigma,I)$ are called
\emph{generalized traces}. Intuitively, any non-decreasing sequence in
$\M$ defines a generalized trace, and two such sequences are
identified whenever they share the same \lub\ in a universal
\lub-completion of~$\M$. In particular, there is a natural embedding
of partial orders $\iota:\M\to\W$ which associates to each trace
$u\in\M$ the generalized trace represented by the constant sequence,
equal to~$u$. In the heap model, generalized traces correspond to
heaps with countably many pieces, either finitely or infinitely many.

By construction, any generalized trace is the \lub\ in $\W$ of a
non-decreasing sequence of traces in~$\M$. Furthermore, $(\W,\leq)$ is
shown to be closed with respect to \lub\ of non-decreasing sequences, and
also to enjoy the following compactness property: for every trace
$u\in\M$ and for every non-decreasing sequence $\seq uk$ in $\M$ such that
$\bigvee\{u_k\tq k\geq0\}\geq u$ holds in~$\W$, there exists an
integer $k\geq0$ such that $u_k\geq u$ holds in~$\M$.  This property
is used to reduce problems concerning generalized traces to problems
concerning traces.

\medskip

The \emph{boundary} of $\M$ is defined as a measurable space
$(\B\M,\FFF)$. The set $\B\M$ is defined by $\B\M=\W\setminus\M$, the
set of \emph{infinite traces}. For any trace $u\in\M$, the
\emph{elementary cylinder of base~$u$} is the non-empty subset of
$\B\M$ defined by
\[
\up u=\{\xi\in\B\M\tq u\leq\xi\}\:;
\]
and $\FFF$ is
the \slgb\ on $\B\M$ generated by the countable collection of all
elementary cylinders.

\subsection{Finite measures on the boundary}
\label{sec:finite-meas-bound}

In this section, we  point out
two basic facts which are valid for any finite
measure on the boundary of a trace monoid~$\M$.

\medskip

First, it is known, see~\cite[p.~150]{bertoni94},  that any two traces
$u,v\in\M$ have a \lub\ $u\vee 
v$ in $\M$ if and only if there exists a trace $w\in\M$ such that
$u\leq w$ and $v\leq w$, in which case $u$ and $v$ are said to be
\emph{compatible}. Using the 
compactness property mentioned above, we deduce:
\begin{equation}
  \label{eq:38}
\forall u,v\in\M\quad  \up u\;\cap\up v=
  \begin{cases}
\up(u\vee v),&\text{if $u$ and $v$ are compatible,}    \\
\emptyset,&\text{otherwise.}
  \end{cases}
\end{equation}

In particular, two different elementary cylinders $\up u$ and $\up v$ may have a
non-empty intersection, even if $u$ and $v$ have the same length. For instance, for the monoid $\M_1=\langle
a,b,c\,|\,a\cdot b=b\cdot a\rangle$, we have the identity $\up
a\;\cap \up b=\up(a \cdot b)$.

It follows from \eqref{eq:38} that the collection of elementary
cylinders, to which is added the empty set, is closed under finite
intersections: this collection forms thus a $\pi$-system, which
generates~$\FFF$.  This implies that a \emph{finite measure on
  $(\B\M,\FFF)$ is entirely determined by its values on elementary
  cylinders.}

\medskip

Second, we highlight a relation satisfied by any finite measure on the
boundary (proof postponed to~\S~\ref{sec:proof-proposition-1}).

\medskip

\begin{proposition}
  \label{prop:9}
  Let $\lambda$ be a finite measure defined on the boundary
  $(\B\M,\FFF)$ of a trace monoid~$\M$. Then:
\begin{equation}
  \label{eq:26}
  \forall u\in\M\qquad
\sum_{c\in\C_\M}(-1)^{|c|}\lambda\bigl(\up(u\cdot c)\bigr)=0\,.
\end{equation}
\end{proposition}

If $\M=\Sigma^*$ is a free monoid, corresponding to the
trivial independence relation $I=\emptyset$\,, then the only non-empty
cliques are the letters of the alphabet~$\Sigma$. In this case, if
$\lambda_k$ is the marginal distribution of $\lambda$ on the words of
length~$k\geq0$, the relation~\eqref{eq:26} is equivalent to
$\lambda_k(u)=\sum_{\alpha\in\Sigma}\lambda_{k+1}(u\cdot\alpha)$\:,
the usual consistency relation between marginals. For general trace
monoids however, the sum in~\eqref{eq:26} contains terms for cliques
of length~$\geq2$. This relates with~\eqref{eq:38}.

\subsection{Valuations and Bernoulli measures}
\label{sec:harm-valu-harm}

Our central object of study is introduced in the following
definition. Throughout the paper, $\RR_+^*$~denotes the set of
positive reals.

\begin{definition}
  \label{def:1}
  Let $\M$ be a trace monoid. We say that a probability measure $\pr$
  on $(\B\M,\FFF)$ is a\/ \emph{Bernoulli measure} if it satisfies:
\begin{gather}
  \label{eq:1}
\forall u\in\M\qquad\pr(\up u)>0\,,\\
\label{eq:2}
\forall u,v\in\M\qquad\pr\bigl(\up(u\cdot v)\bigr)=\pr(\up u)\pr(\up
v)\,.
\end{gather}
The \emph{characteristic numbers} of\/ $\pr$ are defined by
$p_\alpha=\pr(\up\alpha)$ for $\alpha\in\Sigma$\,.
\end{definition}

A Bernoulli measure $\pr$ is entirely characterized by its
characteristic numbers since, by~\eqref{eq:2}, the value of $\pr$ on
all elementary cylinders is determined by the characteristic
numbers. The characteristic numbers appear thus as the natural family
of parameters of a Bernoulli measure.

The main property of Bernoulli measures, condition~\eqref{eq:2},
corresponds to a \emph{memoryless} property on traces. Note that, if
$\M=\Sigma^*$ is a free monoid, then $(\B\M,\FFF)$ is the standard
sample space of infinite sequences with values in~$\Sigma$, and
measures satisfying~\eqref{eq:2} are indeed the standard Bernoulli
measures corresponding to \iid\ processes.

Condition~\eqref{eq:1} is there for convenience and does not involve any loss of
generality. Indeed, it will be satisfied when restricting
ourselves to the sub-monoid generated by those $\alpha\in\Sigma$ such
that $p_\alpha>0$\,.

\medskip
Say that  a function
$f:\M\to\RR_+^*$ which satisfies:
\begin{equation}
  \label{eq:4}
  \forall u,v\in\M\quad f(u\cdot v)=f(u)f(v)\,,
\end{equation}
is a \emph{valuation}, and we insist that $f$ only takes
positive values. The numbers defined by $p_\alpha=f(\alpha)$
are called the \emph{characteristic numbers} of the valuation, and it
is readily seen that for any family of positive
numbers~$(q_\alpha)_{\alpha\in\Sigma}$\,, there exists a unique
valuation with $(q_\alpha)_{\alpha\in\Sigma}$ as characteristic
numbers. By definition, if $\pr$ is a Bernoulli measure
on~$(\B\M,\FFF)$, then the function $u\in\M\mapsto\pr(\up u)$ is a
valuation, that is said to be \emph{induced by\/~$\pr$}\,.

\medskip We recall next the notion of M\"obius polynomial and the
notion of
M\"obius transform of functions. The general notion of M\"obius
transform for partial orders has been introduced by
Rota~\cite{rota64}, and we particularize it to trace monoids.

Considering a trace monoid~$\M$ and any real-valued function
$f:\C\to\RR$\,, the \emph{M\"obius transform} of $f$ is the
function $h:\C\to\RR$ defined by:
\begin{equation}
  \label{eq:5}
  \forall c\in\C\qquad h(c)=\sum_{c'\in\C\tq c\leq c'}(-1)^{|c'|-|c|}f(c')\,.
\end{equation}

By convention, if $f:\M\to\RR_+^*$ is a valuation, the M\"obius transform
of $f$ is defined as the M\"obius transform of its
restriction~$f\rest{\C}$\,.

For each letter $\alpha\in\Sigma$, let $X_\alpha$ be a formal
indeterminate, and let $\ZZ[\Sigma]$ be the ring of polynomials
over~$(X_\alpha)_{\alpha\in\Sigma}$\,. The \emph{multi-variate
  M\"obius polynomial} associated to $(\Sigma,I)$ is
$\mu_\M\in\ZZ[\Sigma]$ defined by:
\begin{equation}
\label{eq:19}
  \mu_\M=\sum_{c\in\C_\M}(-1)^{|c|}
\prod_{\alpha\in c}X_\alpha\,.
\end{equation}

The \emph{evaluation} of the polynomial $\mu_\M$ over a family of real
numbers $(p_\alpha)_{\alpha\in\Sigma}$ is obtained by substituting the
real numbers $p_\alpha$ to the indeterminates $X_\alpha$ in the above
expression. The result is
denoted~$\mu_\M\bigl((p_\alpha)_{\alpha\in\Sigma}\bigr)$\,.

When considering a valuation $f:\M\to\RR_+^*$\,, with characteristic
numbers $(p_\alpha)_{\alpha\in\Sigma}$\,, the M\"obius transform $h$
of $f$ has the following simple expression involving the evaluation of
M\"obius polynomials:
\begin{equation}
  \label{eq:40}
  \forall c\in\C\quad h(c)=f(c)\;\muc\bigl((p_\alpha)_{\alpha\in\Sigma_c}\bigr)\,,
\end{equation}
where $\M_c$ is the sub-monoid defined in~\eqref{eq:39}.  The
expression~\eqref{eq:40} derives immediately from the change of
variable $c'=c\cdot\delta$\,, for $\delta$ ranging over~$\C_{\M_c}$\,,
in the defining sum~\eqref{eq:5} for~$h(c)$, and using the
multiplicative property~\eqref{eq:4}. Two particular instances
of~\eqref{eq:40} shall be noted: for $c=0$, we obtain
$h(0)=\mu_\M\bigl((p_\alpha)_{\alpha\in\Sigma}\bigr)$ since
$f(0)=1$\,, and if $c\in\C_\M$ is maximal then $h(c)=f(c)$ since
$\C_{\M_c}=\{0\}$.

\begin{definition}
  \label{def:3}
  Let $\M=\M(\Sigma,I)$ be a trace monoid. A valuation
  $f:\M\to\RR_+^*$ is a\/ \emph{\goodname\ valuation}\ if its M\"obius
  transform $h:\C_\M\to\RR$ satisfies the following two conditions:
  \begin{align}
    \label{eq:16}
    &(a)\quad h(0)=0\,,&& (b)\quad\forall c\in\Cstar_\M\quad h(c)>0\,.
    \intertext{Equivalently, if $(p_\alpha)_{\alpha\in\Sigma}$ are the
      characteristic numbers of~$f$, then $f$ is a \goodname\
      valuation if and only if:}
  \label{eq:41}
&(a)\quad\mu_\M\bigl((p_\alpha)_{\alpha\in\Sigma}\bigr)=0\,,&&
(b)\quad\forall c\in\Cstar_\M\quad \muc\bigl((p_\alpha)_{\alpha\in\Sigma_c}\bigr)>0\,.
\end{align}
\end{definition}

Our first result transfers the initial problem of determining
Bernoulli measures to the new problem of determining \goodname\
valuations.

\medskip

\begin{theorem}
  \label{thr:1}
  Let $\M(\Sigma,I)$ be an irreducible trace monoid. Then:
\begin{enumerate}
\item\label{item:3} The valuation induced by any Bernoulli measure on
  the boundary of $\M$ is a \goodname\ valuation.
\item\label{item:4} If $f:\M\to\RR_+^*$ is a \goodname\ valuation,
  there exists a unique Bernoulli measure\/ $\pr$ on\/ $(\B\M,\FFF)$
  such that\/ $\pr(\up u)=f(u)$ for all $u\in\M$.
\end{enumerate}
\end{theorem}

Although Theorem~\ref{thr:1} provides valuable information, it does
not state the existence of \goodname\ valuations---and thus of
Bernoulli measures. We will give a positive result on this point
in~\S~\ref{sec:uniform-measures}.

\medskip

The basic relations~\eqref{eq:26}, valid for
any finite measure, when applied to a Bernoulli measure~$\pr$, reduce
to the following:
\[
\forall u\in\M \quad \sum_{c\in\C}(-1)^{|c|}\pr\bigl(\up(
u \cdot c)\bigr)=0\,, \text{ hence } \ \sum_{c\in\C}(-1)^{|c|}\pr(\up
c)=0 \:.
\]
Developing each $\pr(\up c)$ as a product of characteristic
numbers~$p_\alpha$ for $\alpha$ ranging over~$c$, yields the relation
$\mu_\M\bigl((p_\alpha)_{\alpha\in\M}\bigr)=0$\,, proving that
point~$(a)$ in~\eqref{eq:41} is a necessary condition for $\pr$ to be
a Bernoulli measure. This is the only elementary part in the proof of
Theorem~\ref{thr:1}; the rest of the proof is postponed
to~\S~\ref{sec:necess-cond-from} for point~\ref{item:3} and
to~\S~\ref{sec:suff-cond-from} for point~\ref{item:4}.

\subsection{Cartier-Foata subshift and acceptor graph}
\label{sec:cartier-foata-sub}

We introduce now a subshift of finite type based on the Cartier-Foata
decomposition of traces (only elementary notions related to subshifts
will be used, and proper definitions will be recalled when
needed). This is the starting point for a \emph{realization result} in
which Bernoulli measures are described as the law of the trajectories
of a Markov chain.

\medskip Let $\M=\M(\Sigma,I)$ be a trace monoid, and let
$D\subseteq\Sigma\times\Sigma$ be the associated dependence
relation. A pair $(c,c')\in\C_\M\times\C_\M$ of cliques is said to be
\emph{Cartier-Foata admissible}, denoted $c\to c'$, if:
\begin{equation*}
  \forall\beta\in c'\quad\exists\alpha\in c\quad (\alpha,\beta)\in D\,.
\end{equation*}

For every non-empty trace $u\in\M$, there exists a unique integer
$n\geq1$ and a unique sequence of non-empty cliques $(c_1,\ldots,c_n)$
such that $u=c_1\cdot\ldots\cdot c_n$ and $c_i\to c_{i+1}$ holds for
all $i$ in $\{1,\ldots,n-1\}$. This sequence of cliques, denoted
$c_1\to\ldots\to c_n$\,, is called the \emph{Cartier-Foata
  decomposition} or \emph{Cartier-Foata normal form}
of~$u$~\cite{cartier69,viennot86}. The integer $n$ is called the
\emph{height} of the trace~$x$, denoted by $n=\height(x)$.

In the heap interpretation, this sequence of cliques corresponds to
the successive layers of pieces in the heap.

\medskip

The \emph{Cartier-Foata acceptor graph} is the graph $(\Cstar,\to)$.
The \emph{Cartier-Foata subshift} of $\M$ is the set of right-infinite
paths in the graph $(\Cstar,\to)$\,. (It is a "subshift of finite
type" in the terminology of symbolic dynamics.)  See for instance
Figure \ref{fig:carteirfoatacliques} in~\S~\ref{sec:illustr-exampl}
for a concrete example of a Cartier-Foata acceptor graph.  Denote by
$(\Omega,\GGG)$ the measurable space corresponding to the
right-infinite paths in the graph $(\Cstar,\to)$.  Hence elements
$\omega$ of $\Omega$ are given by infinite sequences $(c_k)_{k\geq1}$
of non-empty cliques such that $c_k\to c_{k+1}$ holds for all
$k\geq1$, and $\GGG$ is the \slgb\ of $\Omega$ induced by the product
\slgb, where $\Cstar$ is equipped with the discrete \slgb.

\medskip
The Cartier-Foata decomposition result can be rephrased as the fact
that the finite paths in the graph $(\Cstar,\to)$ are in bijection
with the traces of the monoid. More precisely, for each integer
$k\geq0$, traces of height $k$ are in bijection with paths of length
$k$ in the graph.

In the same way, infinite paths of $(\Cstar,\to)$ correspond naturally
to the points of the boundary of the monoid.  We postpone the proof of
this result to \S~\ref{sec:order-trac-extend} and admit for the time
being that there exists a bi-measurable bijection
$\Psi:\B\M\to\Omega$, which associates to each point of the boundary
$\xi\in\B\M$ an infinite sequence $(c_k)_{k\geq1}$ with values
in~$\Cstar_\M$\,, and entirely characterized by the following two
properties:
\begin{align}
  \label{eq:11}
\forall k\geq1\quad c_k&\to c_{k+1}\,,&
\bigvee_{k\geq1}(c_1\cdot\ldots\cdot c_k)=\xi\,.
\end{align}

The bijection $\Psi:\B\M\to\Omega$ induces a bijection
$\Psi_*:\MM_1(\partial\M,\FFF)\to\MM_1(\Omega,\GGG)$ between the
associated sets of probability measures. We shall always use this
identification.

\medskip

Both spaces $\B\M$ and $\Omega$ come equipped with their own
elementary cylinders: $\up u$~with $u\in\M$ for~$\B\M$\,, and
$\{\omega\in\Omega\tq C_1(\omega)=c_1,\ldots ,C_n(\omega)=c_n\}$ with
$c_1\to\ldots\to c_n$ for~$\Omega$. If $\pr$ and $\prq$ are
probability measures on $\B\M$ and $\Omega$ related by
$\prq=\Psi_*\pr$, the effective correspondence between the values of
$\pr$ and $\prq$ on their respective elementary cylinders is
non-trivial. 
For instance $\pr(\up u)$~differs in general from
$\prq(C_1=c_1,\ldots,C_n=c_n)$ where $c_1\to\ldots\to c_n$ is the
Cartier-Foata decomposition of~$u$. The correspondence
between cylinders is established in details
in~\S~\ref{sec:order-trac-extend}.

\begin{theorem}
  \label{thr:2}
Let $\M$ be an irreducible trace monoid.
\begin{enumerate}
\item\label{item:8} Assume that\/ $\pr$ is a Bernoulli measure on
  $(\B\M,\FFF)$, with $f_\pr(\cdot)=\pr(\up\cdot)$ the induced
  valuation.  Then, under probability\/~$\pr$\,, the sequence
  $(C_k)_{k\geq1}$ of Cartier-Foata cliques  is an
  irreducible and aperiodic Markov chain with values
  in\/~$\Cstar$\,. The law of $C_1$ is the restriction to
  $\Cstar$ of the M\"obius transform $h:\C\to\RR$
  of~$f_\pr$\,, and $h>0$ on~$\Cstar$. The transition matrix of the
  chain is $P=(P_{c,c'})_{(c,c')\in\Cstar\times\Cstar}$ given by:
\begin{align}
     \label{eq:12}
P_{c,c'}&=
\begin{cases}
h(c')/g(c),&\text{if $c\to c'$} \\
  0,&\text{if $\neg(c\to c')$}
\end{cases} &
\text{with\   }g(c)&=\sum_{c'\in\Cstar\tq c\to c'}h(c')\,.
\end{align}

Furthermore, for any integer $n\geq1$, if $c_1,\ldots,c_n$ are $n$
non-empty cliques such that $c_1\to\ldots\to c_n$ holds, then:
 \begin{equation}
   \label{eq:13}
   \pr(C_1=c_1,\ldots,C_n=c_n)=f_\pr(c_1)\cdots f_\pr(c_{n-1})h(c_n)\,.
 \end{equation}

\item\label{item:9} Conversely, let $f:\M\to\RR_+^*$ be a \goodname\ valuation,
  and let $h:\C\to\RR$ be the M\"obius transform of~$f$. Then the
  restriction $h\rest\Cstar$ is a probability distribution
  on~$\Cstar$. The Markov chain on\/ $\Cstar$ with $h\rest\Cstar$ as
  initial law, and with transition matrix $P$ given as
  in~\eqref{eq:12} above, induces a Bernoulli measure $\pr$ on $\B\M$
  which satisfies\/ $\pr(\up u)=f(u)$ for all traces $u$ in $\M$.
\end{enumerate}
\end{theorem}

It is worth observing that $h\rest\Cstar$ is \emph{not} the stationary
distribution of $P$, implying that $(C_n)_{n\geq1}$ is not stationary
with respect to $n$ under~$\pr$.  Markovian measures with the property
\eqref{eq:13} also appear in the context of random walks on some
infinite groups where they are called ``Markovian
multiplicative''~\cite{mair04}.

The proof of Theorem~\ref{thr:2} is postponed to
\S~\ref{sec:necess-cond-from} for point~\ref{item:8} and to
\S~\ref{sec:suff-cond-from} for point~\ref{item:9}.

\subsection{Uniform measures}
\label{sec:uniform-measures}

So far we have obtained polynomial normalization conditions for the
characteristic numbers of Bernoulli measures
(\S~\ref{sec:harm-valu-harm}) and we have identified Bernoulli
measures with certain Markov measures on a combinatorial subshift
(\S~\ref{sec:cartier-foata-sub}). The reader might have noticed that
the actual existence of Bernoulli measures has not yet been stated.

In this subsection we state the existence of \emph{uniform Bernoulli}
measures, those having all their characteristic numbers identical. We
also introduce the weaker notion of uniform measure. An equivalence
between uniform measures and uniform Bernoulli measures is stated---a
non-trivial result. Then we show how small deformations of the
characteristic numbers around the particular value for the uniform
measure lead to a continuum of distinct Bernoulli measures.

\medskip Let $\M$ be an irreducible trace monoid, and assume there
exists a Bernoulli measure for which all characteristic numbers are
equal, say to some real~$p>0$. Then, according to Theorem~\ref{thr:1}
point~\ref{item:3}, and using the formulation stated
in~\eqref{eq:41}--$(a)$, the number $p$ must be a root of the
\emph{single-variable M\"obius polynomial} $\mu_\M(X)\in\ZZ[X]$ defined
by:
\begin{equation}
  \label{eq:10}
  \mu_\M(X)=\sum_{c\in\C}(-1)^{|c|}X^{|c|}\,.
\end{equation}

We therefore face two questions. First, among the roots of
$\mu_\M(X)$, which ones correspond indeed to a Bernoulli measure? Such
measures, and we shall prove their existence, we call \emph{uniform
  Bernoulli measures}.

Obviously, any uniform Bernoulli measure satisfies the following property:
  \begin{equation}
    \label{eq:28}
    \forall u,v\in\M\qquad |u|=|v|  \implies \pr(\up u)=\pr(\up v)\,.
  \end{equation}
  We emphasize that the above property is purely metric.  Say that a
  probability measure on~$\B\M$ satisfying~\eqref{eq:28} is
  \emph{uniform}. The second question is: does any uniform measure
  belong to the class of Bernoulli measures?  In other words,
  does~\eqref{eq:28} imply the memoryless property
  \mbox{$\pr\bigl(\up(u\cdot v)\bigr)=\pr(\up u)\pr(\up v)$}\,? Note
  that the answer is clearly affirmative in the case of a free monoid,
  but much less trivial for a trace monoid.

Next theorem brings answers to the two above questions. The statement
requires the knowledge of the following fact, which will be given in
an even more precise form below in Theorem~\ref{thr:5}: \emph{the
  M\"obius polynomial of an independence pair $(\Sigma,I)$ has a
  unique root of smallest modulus. This root is real and lies in the
  open interval $(0,1)$}.

\begin{theorem}
\label{thr:6}
  Let $\M$ be an irreducible trace monoid, and let $p_0$ be the unique
  root of smallest modulus of the M\"obius polynomial~$\mu_\M(X)$. Then:
  \begin{enumerate}
  \item\label{item:12} There exists a unique uniform Bernoulli
    measure\/ $\pr_0$ on\/ $(\B\M,\FFF)$. It is entirely characterized
    by\/ $\pr_0(\up u)=p_0^{|u|}$.
  \item\label{item:13} Any uniform measure is Bernoulli
    uniform. Hence\/ $\pr_0$ is also the unique uniform measure
    on~$(\B\M,\FFF)$.
  \end{enumerate}
\end{theorem}

Point \ref{item:13} in Theorem \ref{thr:6} appears as a confirmation
of the central role of Bernoulli measures.

Having identified at least one Bernoulli measure for each irreducible
trace monoid, we are able to construct many others by considering
small variations around the value $(p_0,\ldots,p_0)$ for the family of
characteristic numbers.

\begin{proposition}
  \label{prop:7}
  Let $\M=\M(\Sigma,I)$ be irreducible with
  $|\Sigma|>1$. There exists a continuous family of different
  Bernoulli measures on~$\B\M$\,.
\end{proposition}

The proofs of Theorem~\ref{thr:6} and of Proposition~\ref{prop:7} are
postponed to~\S~\ref{sec:unif-meas-exist}.

\section{Examples and Applications}
\label{sec:exampl-appl}

\subsection{Two illustrative examples}
\label{sec:illustr-exampl}

We illustrate the above results on two specific examples.  

We first concentrate on $\M_1=\bigl\langle a,b,c\;|\;a\cdot b=b\cdot
a\bigr\rangle$, and provide a complete description of associated
Bernoulli measures.  Non-empty cliques of $\M_1$ are given by
$\{a,b,c,a\cdot b\}$, and the Cartier-Foata acceptor graph is depicted
in Figure~\ref{fig:carteirfoatacliques}.
\begin{figure}[t]
\centerline{\SelectTips{cm}{}
\xymatrix@C=4em@R=1.6em{
&\etat{a\cdot b}\ar@/_/[dl]\ar@/^/[dr]\ar@/^/[dd]\POS!R(.2)!U(.5)\ar@(ur,ul)[]!L(.2)!U(.5)
\\
\etat{a}\ar@/_/[dr]\POS!U(.2)!L(.5)\ar@(ul,dl)[]!D(.2)!L(.5)&&
\etat{b}\ar@/^/[dl]\POS!D(.2)!R(.5)\ar@(dr,ur)[]!U(.2)!R(.5)
\\
&\etat{c}\ar@/_/[ul]\ar@/^/[ur]\ar@/^/[uu]\POS!L(.2)!D(.5)\ar@(dl,dr)[]!R(.2)!D(.5)
}}\par\bigskip\medskip
\caption{\textsl{Cartier-Foata acceptor graph of $\M_1=\langle
    a,b,c\;|\;a\cdot b=b\cdot a\rangle$}}
\label{fig:carteirfoatacliques}
\end{figure}
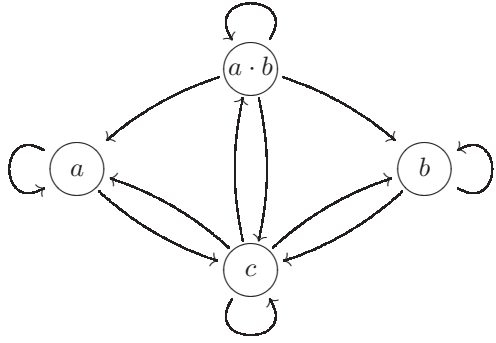
 The M\"obius polynomials of
$\M_1$ is $ \mu_{\M_1}(X)=1-3X+X^2$\,, with roots
$(3\pm \sqrt5)/2$\,. Keeping the same symbols $a,b,c$ to
denote the characteristic numbers of a generic valuation
$f:\M_1\to\RR_+^*$, the M\"obius transform $h$ of $f$ is given as
follows:
\begin{equation*}
  \begin{array}{l|ccccc}
    \text{clique $\gamma$}&0&a&b&c&a\cdot b\\
\text{valuation $f(\gamma)$}&1&a&b&c&ab\\
\text{M\"obius transform $h(\gamma)$}&1-a-b-c+ab&a-ab&b-ab&c&ab
  \end{array}
\end{equation*}

\Goodname\ valuations are thus determined, according to
Definition~\ref{def:3}, by the following conditions on parameters:
\begin{align}
\label{eq:15}
  1-a-b-c+ab&=0\,,&
a(1-b)&>0\,,&
b(1-a)&>0\,,&c&>0\,,&ab&>0\,.
\end{align}
 
It follows from Theorem~\ref{thr:1} that Bernoulli measures on $\M_1$
are in bijective correspondence with the set of triples
$(a,b,c)\in(\RR_+^*)^3$ solutions of~\eqref{eq:15}.  The set of
admissible triples forms a surface, a plot of which is given in
Figure~\ref{fig:piojojhkjnsz}.

\begin{figure}[hb]
  \centering
\input{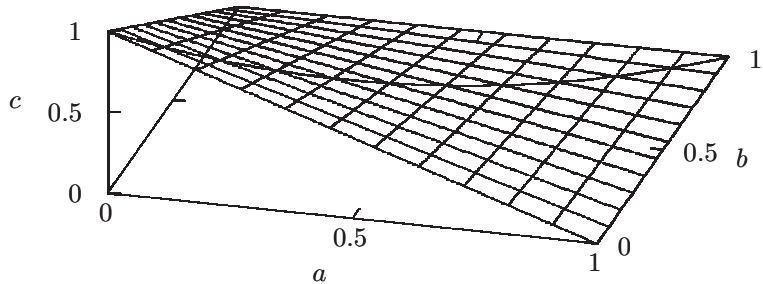}
  \caption{\textsl{Plot of the possible values for the
      characteristic numbers of a Bernoulli measure on~$\M_1$}}
\label{fig:piojojhkjnsz}
\end{figure}

We can easily compute the transition matrix $P$ associated with a
generic Bernoulli measure with parameters $(a,b,c)$ solution
of~\eqref{eq:15}. The normalization factor $g$ from
Theorem~\ref{thr:2} is equal to $1$ on all maximal cliques, which are
$a\cdot b$ and~$c$. We observe that $g(a)=a-ab+c=1-b$, taking into
account that $1-a-b-c+ab=0$. Similarly, $g(b)=1-a$. According to
formula~\eqref{eq:12} stated in Theorem~\ref{thr:2}, and indexing the
rows and columns of the transition matrix according to the cliques
$(a,b,c,a\cdot b)$ in this order, we get:
\begin{gather*}
  P=
  \begin{pmatrix}
a&0&1-a&0\\    
0&b&1-b&0\\
a-ab&b-ab&c&ab\\
a-ab&b-ab&c&ab
  \end{pmatrix}
\end{gather*}

It is readily checked by hand on this example, and this is true in
general, that conditions~\eqref{eq:15} ensure that the above matrix
has all its entries non negative and that all lines sum up to~$1$.

According to Theorem~\ref{thr:6}, the only uniform measure $\pr_0$ on
$\M_1$ is determined by the root $p_0=(3-\sqrt5)/2=0.382\cdots$
of~$\mu_{\M_1}$\,, and satisfies $\pr_0(\up u) = p_0^{|u|}$ for all
$u$ in $\M$.  Note that, for this example, the other root of
$\mu_{\M_1}$ is outside $(0,1)$, so it is immediate that only $p_0$
can correspond to a probability. But the M\"obius polynomial might
have several roots within $(0,1)$, as the next example reveals.

\medskip

Consider the trace monoid $\M_2=\M(\Sigma_2,I_2)$ with
$\Sigma_2=\{a_1,\ldots,a_5\}$, and which associated dependence
relation $D_2$ is depicted in Figure~\ref{fig:dependas}.
\begin{figure}[htb]
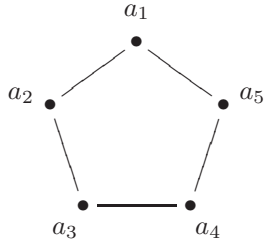

  \begin{equation*}
\xy{
\xypolygon5"A"{~:{(2,2):(3,3):}~={0}{\hbox{\strut$\;\bullet\;$}}}%
\xypolygon5"B"{~:{(2,2):(4,4):}~><{@{}}~={0}{a_{\xypolynode}}}
}%
\endxy
  \end{equation*}
  \caption{\textsl{Dependence relation $D_2$ for $\M_2$\,. Self-loops
      are omitted.}}
  \label{fig:dependas}
\end{figure}
Then the M\"obius polynomial is $\mu_{\M_2}(X)=1-5X+5X^2$\,, with
roots $q_0=1/2 - \sqrt5/10$ and
$q_1=1/2 + \sqrt5/10$\,.  Hence, $\mu_{\M_2}$~has its two
roots within $(0,1)$, so we need to use the full statement of
Theorem~\ref{thr:6} point~\ref{item:12} in order to rule out $q_1$ and
retain only $q_0$ as defining a uniform measure.

This can be double-checked by hand on this example. Consider the
valuation $f(u)=q_1^{|u|}$\,. Let $h$ be the M\"obius transform
of~$f$. We have, for $i\in \{1,\dots ,5\}$: $h(a_i)=
q_1(1-2q_1)=-q_1/\sqrt5 <0$\,.  Therefore, the valuation $f$ is not
\goodname\ and $q_1$ does not qualify to define a uniform measure.

\subsection{Computing the speedup}
\label{sec:comp-speed-trace}

In a trace monoid, what is the ``average parallelism''? Or what is the
average speedup of the parallel execution time compared to the
sequential one? Or what is the average density of a heap? The
questions are natural and have been extensively
studied~\cite{saheb89,CePe,krob03,bertoni08}. Obviously the
probability assumptions need to be specified for the questions to make
sense.

In the remaining of~\S~\ref{sec:comp-speed-trace}, we consider an
irreducible trace monoid~$\M$.

Let $\height(u)$ denote the \emph{height} of a trace $u\in\M$, which
is equivalently defined (see also~\S~\ref{sec:cartier-foata-sub})
either as the number of cliques in the Cartier-Foata decomposition
of~$u$,
or as the height of the heap of
pieces associated with~$u$, and can also be interpreted as the
``parallel execution time'' of~$u$.

One standard approach is to define the average parallelism as the
quantity $\lambda_{\M}= \lim_{n\rightarrow \infty} n/\height(x_1\cdot
x_2 \cdot\ldots\cdot x_n)$ where $(x_n)_n$ is an independent and
uniformly distributed sequence of $\Sigma$-valued random
variables. The limit exists indeed and is constant with probability
one, by a sub-additivity argument~\cite{saheb89}. The quantity
$\lambda_{\M}$ is non-algebraic except for small trace
monoids~\cite{krob03}, and is NP-hard even to approximate~\cite{BGT}.

Another approach, more closely related to our probabilistic model, is
as follows. For all integer $n\geq 0$, set $\M_n = \{u\in \M \tq |u|
=n \}$ and let $\nu_n$ be the uniform probability distribution on the
finite set~$\M_n$\,.  Let $\varphi:\M\to\RR_+$ be the
function defined by:
\begin{align*}
\varphi(0)&=1\,,& x\neq0\quad \varphi(x)&=\frac{|x|}{\height(x)}\,.
\end{align*}

Clearly, $1\leq\varphi(\cdot)\leq K$ holds for $K$ the maximal size of
cliques.  Define, assuming existence, the two limits:
\begin{align}
\label{eq-invspeed} 
  \speed_{\M} & =  \lim_{n\rightarrow \infty}  \esp_{\nu_n}1/\varphi(\cdot)  =  \lim_{n\rightarrow \infty} 
\frac{1}{\# \M_n} \sum_{u\in \M_n}
  \frac{\height(u)}{n} \\
\label{eq-speed}
\rho_{\M} & =  \lim_{n\rightarrow \infty}  \esp_{\nu_n}\varphi(\cdot)   =   \lim_{n\rightarrow \infty} 
\frac{1}{\# \M_n} \sum_{u\in \M_n}\frac{n}{\height(u)}   
\end{align}

The quantities $\rho_{\M}$ and $\speed_{\M}$ have been studied under
the respective names of \emph{speedup} and \emph{average
  parallelism}. In~\cite{krob03}, the limit in~\eqref{eq-invspeed} is
shown to exist and to be an algebraic number (the notation is
``$\lambda_{\mathbb{M}}$'' in~\cite{krob03}). The method is based on
manipulations of the bi-variate generating series
$L(x,y)=\sum_{u\in\M}x^{|u|}y^{\height(u)}$\,. The expression obtained
for $\speed_{\M}$ is a ratio:
\begin{gather}
\label{eq:22}
\speed_{\M} =  \frac{\bigl[(\partial L/\partial y )(x,1)\cdot ({p_0}
  -x)^2\bigl]_{|x={p_0}}}{{p_0}\cdot \bigl[L(x,1)\cdot ({p_0}
  -x)\bigl]_{|x={p_0}}} \,,
\end{gather}
where $p_0$ is the root of smallest modulus of the M\"obius polynomial
of~$\M$. The above expression is tractable to a certain extent since the series
$L(x,y)$ is shown to be rational~\cite[Prop. 4.1]{krob03}.

\medskip

Using the results of the present paper, we obtain much more. Let $\pr$
be the uniform measure on~$\B\M$. Recall that the sequence
$(C_n)_{n\geq1}$ of cliques is then an irreducible and aperiodic
Markov chain on~$\Cstar$. Let $\pi$ be the stationary measure of this
chain, that is to say, $\pi$~is the unique probability vector on
$\Cstar$ satisfying: $\pi P=\pi$, where $P$ is the transition matrix
given in Theorem~\ref{thr:2}.

The distribution $\nu_n$ on $\M_n$ induces a distribution of the
ratios $\varphi(x)$ on~$[1,K]$, as well as a distribution of the
inverses $1/\varphi(x)$ on $[1/K,1]$. We denote respectively
by~$\varphi_*\nu_n$\,, and by~$(1/\varphi)_*\nu_n$\,, the induced
distributions.

\begin{proposition}
The limits in\/ \eqref{eq-invspeed}--\eqref{eq-speed} exist. They 
are algebraic numbers, and satisfy the following formulas: 
\begin{align}
\label{eq:7}
      \rho_\M&=\sum_{c\in\Cstar}\pi(c)|c|\,,& \speed_{\M}&= \frac1{\rho_{\M}}\,. 
\end{align}

Furthermore, the sequences of probability distributions
$(\varphi_*\nu_n)_{n\geq1}$ and\linebreak
$((1/\varphi)_*\nu_n)_{n\geq1}$ converge weakly respectively towards
the Dirac measure~$\delta_{\rho_\M}$\,, and toward the Dirac
measure~$\delta_{\speed_\M}$\,.
\end{proposition}

The formula in \eqref{eq:7} provides a more tractable expression for
$\speed_{\M}$ than the one in~\eqref{eq:22}. The concentration results
are new. In a forthcoming work based on spectral methods, the first
author and S.~Gou\"ezel show the weak convergence of $\sqrt
n(1/\varphi(\cdot)-\gamma_\M)$ towards a normal law $\mathcal
N(0,\sigma^2)$, non degenerated unless $\M$ is a free monoid.

\begin{proof}
  We briefly sketch the proof.  Set $\rho=\sum_{c\in \Cstar} \pi(c)
  |c|$. Clearly, $\rho$~is algebraic since $p_0$ is algebraic, and
  since the coefficients $\pi(c)$ are solutions of a linear system
  involving only $p_0$ and its powers.

For every integer $n\geq1$, let
$Y_n=C_1\cdot\ldots\cdot C_n$\,, of height $n$ by construction, and
let:
\begin{align*}
  A_n&=Y_{P_n}\,,&P_n&=\max\{q\in\NN\tq|Y_q|\leq n\}\,.
\end{align*}

Clearly, $P_n\geq n/K$ holds and thus
$\lim_{n\to\infty}P_n=\infty$\,. By the Ergodic theorem for Markov
chains (see~\cite[Part.~I, \S~15, Theorem 2]{chung60}, or
\cite[Theorem~4.1 p.~111]{brem}), the ratios
$\varphi(Y_n)=|Y_n|/\height(Y_n)=(|C_1|+\ldots+|C_n|)/n$ converge
\pas\ towards~$\rho$, and so do the ratios
$\varphi(A_n)=\varphi(Y_{P_n})$\,.

Using the expression~\eqref{eq:13} of Theorem~\ref{thr:2} on the one
hand, and Proposition~\ref{prop:4} below on the other hand, we see
that there is a constant $C_1>0$ such that $\pr(A_n=y)\geq C_1p_0^n$
for all traces $y$ belonging to the range of values of~$A_n$\,. Hence,
by the asymptotics $p_0^{-n}\sim C(\#\M_n)$ recalled in
Theorem~\ref{thr:5} below, there is a constant $C_2>0$ such that
$\pr(A_n=y)\geq C_2/(\#\M_n)$\,, for all such traces~$y$.

Then, let $\varepsilon>0$. We have:
\begin{align*}
  \nu_n\bigl(\varphi(x)\notin[\rho-\varepsilon,\rho+\varepsilon])&=
\sum_{y\in\M_n}\frac1{\#\M_n}\un_{\{\varphi(y)\notin[\rho-\varepsilon,\rho+\varepsilon]\}}\\
&\leq\frac1{C_2}\sum_{y\in\M_n}\un_{\{\varphi(y)\notin[\rho-\varepsilon,\rho+\varepsilon]\}}\pr(A_n=y)\\
&\leq\frac1{C_2}\pr\bigl(\varphi(A_n)\notin[\rho-\varepsilon,\rho+\varepsilon]\bigr)\to_{n\to\infty}0
\end{align*}
since the convergence $\varphi(A_n)\to\rho$ holds $\pr$-almost surely,
and thus in distribution. This proves the convergence in distribution
of $\varphi(x)$ towards~$\delta_\rho$\,. The convergence in
distribution of $1/\varphi(x)$ towards $\delta_{1/\rho}$ follows. And
this implies of course the convergence of $\esp_{\nu_n}\varphi(\cdot)$
towards~$\rho$, and of $\esp_{\nu_n}1/\varphi(\cdot)$
towards~$1/\rho$.  It follows that $\rho_{\M}=\rho$ and
$\speed_{\M}=1/\rho$.
\end{proof}

\smallskip

For instance, for the trace monoids $\M_1$ and $\M_2$
analyzed in~\S~\ref{sec:illustr-exampl}, we obtain, after computation
of the associated stationary measure~$\pi$\,:
\begin{align*}
  \rho_{\M_1} &= 5(7 -\sqrt{5}) /22 = 1.0827 \cdots&
  \rho_{\M_2}&= (29 - \sqrt{5}) /22 = 1.2165 \cdots \\
\gamma_{\M_1}&=(7+\sqrt5)/10=0.924\cdots&
\gamma_{\M_2}&=(29+\sqrt5)/38=0.822\cdots
\end{align*}
consistently with the results of~\cite[Appendix~B]{krob03}, where the
case of $\M_1$ is included in the table.  So parallelism increases the
speed of execution by about $8\%$ in the monoid $\M_1$ and by about
$22\%$ in the monoid $\M_2$.

\section{Auxiliary tools}
\label{part:auxiliary-tools}

This section introduces auxiliary tools needed for the proofs of the
previously stated results.

\subsection{Elementary cylinders and sequences of cliques}
\label{sec:order-trac-extend}

Let $\M=\M(\Sigma,I)$ be a trace monoid. In this subsection, we
establish the correspondence between points of the boundary~$\B\M$,
and infinite sequences of non-empty cliques satisfying the
Cartier-Foata condition. We also describe how the order on traces
transposes to their Cartier-Foata normal form.

\medskip

To each non-empty trace $u\in\M$ of Cartier-Foata normal form
$d_1\to\ldots\to d_n$\,, we associate an infinite sequence
$(c_k)_{k\geq1}$ of cliques, defined as follows: $c_k=d_k$ if $1\leq
k\leq n$, and $c_k=0$ for $k>n$. The sequence $(c_k)_{k\geq1}$ is
called the \emph{extended Cartier-Foata decomposition} of~$u$,
abbreviated \ECF.

Recall from~\S~\ref{sec:cartier-foata-sub} that~$\height(u)$, the
height of the trace~$u$, is defined as the number of cliques in the
Cartier-Foata normal form of~$u$.  Equivalently, $\height(u)$~is the
number of non empty cliques in the \ECF\ decomposition of~$u$.

The Cartier-Foata decomposition establishes, for each integer
$n\geq1$, a bijection between the set of traces of height~$n$ and a
subset of the product~$\C^n$\,. How the ordering between traces is
read on their Cartier-Foata decompositions is the topic of next
results. In particular, the obtained order on Cartier-Foata sequences
is strictly coarser in general than the order induced by the product
order on~$\C^n$\,.

In the following result and later on, we make use of the notion of
parallel cliques introduced in~\S~\ref{sec:trace-monoids-their}, by
writing $c\indep c_1,\ldots,c_n$ if $c\indep c_i$ for all integers
$i\in\{1,\ldots,n\}$.

\begin{lemma}
  \label{prop:1}
  Let $\M=\M(\Sigma,I)$ be a trace monoid and let $u,v\in\M$ be two
  non-empty traces. Let $c_1\to\ldots\to c_n$ and $d_1\to\ldots\to
  d_p$ be the Cartier-Foata decompositions of $u$ and of~$v$. Then
  $u\leq v$ if and only $n\leq p$ and there are $n$ cliques
  $\gamma_1,\ldots,\gamma_n$ such that, for all $i\in\{1,\ldots,n\}$:
\begin{enumerate}
\item\label{item:10} $\gamma_i\indep c_i,\ldots,c_n$\,; and
\item\label{item:11} $d_i=c_i\cdot\gamma_i$\,.
\end{enumerate}
\end{lemma}

\begin{proof}
  If the Cartier-Foata normal forms of $u$ and $v$ satisfy the
  properties stated in points~\ref{item:10}--\ref{item:11}, then an easy
  induction argument shows that:
\begin{gather*}
v=  c_1\cdot\ldots\cdot c_n\cdot(\gamma_1\cdot\ldots\cdot \gamma_n)
    \cdot d_{n+1}\cdot\ldots\cdot d_p \,,
\end{gather*}
which implies that $u\leq v$. 

Conversely, the proof is simple using the heap of pieces
intuition. Here is the main argument. Assume that $u\leq v$, and let
$w$ be such that $v=u\cdot w$. If $w=0$, the result is
trivial. Otherwise, let $\delta_1\to\ldots\to\delta_r$ be the
Cartier-Foata normal form of~$w$. Apply the following recursive
construction: pick a letter $\alpha\in\delta_1$, and move $\alpha$
from $\delta_1$ to the clique $c_i$ where $i$ is the smallest index such
that $\alpha\indep c_i,\dots , c_n$. If there is no such index~$i$,
the letter $\alpha$ stays in~$\delta_1$\,. Then repeat the operation,
until all letters of $\delta_1$ have been dispatched. Once this is
done, recursively apply the same procedure to~$\delta_2$ up to
$\delta_r$\,. Some cliques among the $\delta_i$ might entirely vanish
during the procedure. The whole procedure yields the Cartier-Foata
normal form of $u\cdot w$\,, under the requested form.
\end{proof}

For each integer $p\geq0$, we define the $p$-cut operation as the
mapping $\cut_p:\M\to\M$, $u \mapsto\cut_p(u)=c_1\cdot\ldots\cdot
c_p$\,, where $(c_k)_{k\geq1}$ is the \ECF\ decomposition of~$u$.

\begin{corollary}
  \label{cor:1}
Let $\M$ be a trace monoid, and let $u,v\in\M$ be two traces. Then
$u\leq v$ if and only if $u\leq\cut_{\height(u)}(v)$\,.
\end{corollary}

\begin{proof}
  If $u\leq\cut_{\height(v)}(v)$, then $u\leq v$ since $\cut_n(v)\leq v$
  holds trivially for any integer $n\geq0$. Conversely, assume that $u\leq
  v$. Then the Cartier-Foata normal forms $c_1\to\ldots\to c_n$ and
  $d_1\to\ldots \to d_p$ of $u$ and $v$ satisfy properties 1 and 2 in
  Lemma~\ref{prop:1}. We have 
  $\height(u)=n$ and 
  \begin{equation*}
    \cut_{\height(u)}(v)=(c_1\gamma_1)\cdot\ldots\cdot(c_n\gamma_n)=(c_1\cdot\ldots\cdot c_n)
\cdot\gamma_1\cdot\ldots\cdot \gamma_n\geq u\,.
  \end{equation*}
The proof is complete.
\end{proof}

\begin{corollary}
  \label{cor:2}
  Let $\M$ be a trace monoid. Let $(c_k)_{k\geq1}$ and
  $(d_k)_{k\geq1}$ be the \ECF\ decompositions of two traces $u$
  and~$v$ of $\M$.  If $u\leq v$, then $c_k\leq d_k$ for all integers
  $k\geq1$.
\end{corollary}

\begin{proof}
  The inequality $c_k\leq d_k$ is trivial if $c_k=0$\,. And for
  $c_k\neq 0$, then $c_k$ belongs to the Cartier-Foata normal form
  of~$u$, and $c_k\leq d_k$ follows from Lemma~\ref{prop:1}
  point~\ref{item:11}.
\end{proof}

We lift the \ECF\ decomposition to generalized traces as follows.

\begin{lemma}
  \label{prop:2}
  Let $\M=\M(\Sigma,I)$ be a trace monoid. Then, for every generalized
  trace $u\in\W(\Sigma,I)$, there exists a unique infinite sequence of
  cliques $(c_k)_{k\geq1}$ such that
  $u=\bigvee_{k\geq1}(c_1\cdot\ldots\cdot c_k)$ and $c_k\to c_{k+1}$
  holds for all integers $k\geq0$\,.
\end{lemma}

\begin{proof}
  The result is clear if $u\notin\B\M$.  And if $u\in\B\M$, we
  consider a non-decreasing sequence $\seq un$ in $\M$ such that
  $u=\bigvee_{n\geq0}u_n$\,. Then, if $(c_{n,k})_{k\geq1}$ is the
  \ECF\ decomposition of~$u_n$\,, it follows from
  Corollary~\ref{cor:2} that $(c_{n,k})_{n\geq0}$ is non-decreasing
  in~$\C$ for each integer $k\geq1$, and thus eventually constant, say
  equal to~$c_k$\,. Routine verifications using the compactness
  property stated in~\S~\ref{sec:trace-monoids-their} show that the
  sequence $(c_k)_{k\geq1}$ thus defined is the only adequate
  sequence.
\end{proof}

Say that the sequence $(c_k)_{k\geq1}$ associated to a generalized
trace $u\in\W(\Sigma,I)$ as in Lemma~\ref{prop:2} is the
\emph{\ECF\ decomposition} of~$u$.  For each $\xi\in\B\M$, the
sequence $(c_k)_{k\geq1}$ is the unique sequence of non-empty cliques
announced in~\eqref{eq:11}.

Recall that we have defined $\Omega$ as the set of infinite paths in
the graph $(\Cstar,\to)$ of non-empty cliques.  Mapping each point
$\xi\in\B\M$ to its \ECF\ decomposition defines a well-defined
application $\Psi:\B\M\to\Omega$, which is the mapping announced
in~\S~\ref{sec:cartier-foata-sub}. It is bijective; its inverse is
given, if $\omega\in\Omega$ has the form
$\omega=(c_k)_{k\geq1}$\,, by:
\begin{equation*}
  \Psi^{-1}(\omega)=\bigvee_{k\geq1}(c_1\cdot\ldots\cdot c_k)\,.
\end{equation*}

The following result, which is based on Lemma \ref{prop:1}, explores
how elementary cylinders transpose through the \ECF\ decomposition,
connecting the two different points of view on the boundary elements:
the \emph{intrinsic} point of view through elementary cylinders, and
the \emph{effective} point of view through sequences of cliques.

\begin{proposition}
  \label{prop:3}
  Let $\M$ be a trace monoid. For any element $\xi\in\B\M$, denote
  by $\bigl(C_k(\xi)\bigr)_{k\geq1}$ the \ECF\ decomposition
  of~$\xi$. Let $n\geq1$ be an integer, and let $c_1,\ldots,c_n\in\C$ be
  $n$ cliques such that $c_1\to\ldots\to c_n$ holds. Put
  $v=c_1\cdot\ldots\cdot c_{n-1}$ and $u=v\cdot c_n$\,.  Then the
  following equalities of subsets of\/ $\B\M$ hold:
\begin{gather}
  \label{eq:8}
\up u = \up(c_1\cdot\ldots\cdot c_n)=\bigl\{\xi\in\B\M\tq C_1(\xi)\cdot\ldots\cdot
C_n(\xi)\geq u\bigr\}\,,\\
\label{eq:9}
\bigl\{\xi\in\B\M\tq C_1(\xi)=c_1\,,\ldots,\, C_n(\xi)=c_n\bigr\}=
\up u\setminus\Bigl(
\bigcup_{\substack{c\in\C\,:\\c_n<c}}\up(v\cdot c)\Bigr)\,,
\end{gather}
where $c_n<c$ means $c_n\leq c$ and $c_n\neq c$.
\end{proposition}

\begin{proof}
   Put $\W=\W(\Sigma,I)$, and extend the cut
operations $\cut_p:\M\to\M$ defined above for all integers $p\geq0$,
to mappings $\cut_p:\W\to\M$ in the obvious way. Combining the
compactness property with Corollary~\ref{cor:1} yields:
\begin{equation}
\label{eq:6}
  \forall u\in\M\quad\forall v\in\W\quad
u\leq v\iff u\leq\cut_{\height(u)}(v)\,.
\end{equation}

Applied to $u=c_1\cdot\ldots\cdot c_n$ as in the statement and to
$\xi\in\B\M$ in place of~$v$, this is~(\ref{eq:8}).

We now prove~(\ref{eq:9}). Set:
\begin{align*}
  A&=\bigl\{\xi\in\B\M\tq C_1(\xi)=c_1\wedge\ldots\wedge
  C_n(\xi)=c_n\bigr\}\,,&B&=\bigcup_{c\in\C\tq c_n<c}\up(v\cdot c)\,.
\end{align*}

It is obvious that $A\subseteq \up u$. We prove that $A\cap
B=\emptyset$. For this, by contradiction, assume there exists $\xi\in
A\cap B$, and let $(\delta_k)_{k\geq1}$ be the \ECF\ of~$\xi$. Then
$\delta_i=c_i$ for all $i\in\{1,\ldots,n\}$ since $\xi\in A$. Let
$c\in\C$ be a clique such that $c_n<c$ and $\xi\in\up(v\cdot
c)$. Clearly, $\tau(v\cdot c)=n$. Applying~(\ref{eq:6})
 to $v\cdot c\leq\xi$ we get thus $c_1\cdot\ldots\cdot c_{n-1}\cdot c\leq
 c_1\cdot\ldots\cdot c_{n-1}\cdot c_n$\,, and then by left
 cancellativity of the monoid, $c\leq c_n$, a contradiction. This
 proves that $A\cap B=\emptyset$, and thus the $\subseteq$ inclusion
 of~(\ref{eq:9}). 

For the converse $\supseteq$ inclusion, let $\xi\in\up u\setminus B$,
keeping the notation  $(\delta_k)_{k\geq1}$ for its \ECF\
decomposition. Since $\height(u)=n$, it follows from~(\ref{eq:6}) that
$c_1\cdot\ldots\cdot c_n\leq\delta_1\cdot\ldots\cdot\delta_n$\,. Hence
$\delta_i=c_i\cdot\gamma_i$ for some cliques
$\gamma_1,\ldots,\gamma_n$ as in Lemma~\ref{prop:1}. Using the
properties of the cliques~$\gamma_i$'s, we have
$\xi\geq\delta_1\cdot\ldots\cdot\delta_n=c_1\cdot\ldots\cdot
c_n\cdot\gamma_1\cdot\ldots\cdot\gamma_n$\,. Since $\xi\notin B$ by
assumption, this implies that $\gamma_i=0$ for all
$i\in\{1,\ldots,n\}$, and thus $\xi\in A$.
\end{proof}

\begin{corollary}
  The bijection $\Psi:\B\M\to\Omega$ which associates to each point
  $\xi\in\B\M$ its \ECF\ decomposition, is bi-measurable
  with respect to $(\B\M,\FFF)$ and $(\Omega,\GGG)$.
\end{corollary}

\begin{proof}
  The fact that $\Psi$ is measurable follows from~(\ref{eq:9}). The
  fact that $\Psi^{-1}$ is measurable follows from~(\ref{eq:8}).
\end{proof}

\subsection{Generating series and asymptotics}
\label{sec:prel-comb-results}

For   $\M$ a trace
monoid and $k\geq0$ an integer, 
let $\lambda_\M(k)$ be the number of traces of length $k$ :
\begin{equation*}
  \lambda_\M(k)=\#\{u\in\M\tq |u|=k\}\,.
\end{equation*}
Let $G_\M(X)$ be the generating series
  of~$\M$, defined by:
\begin{equation*}
  G_\M(X)=\sum_{u\in\M}X^{|u|}=\sum_{k\geq0}\lambda_\M(k)X^k\,.
\end{equation*}

The following result is standard, and is the basis of the
combinatorial study of trace monoids \cite{cartier69,viennot86,krob03,goldwurm00}.

\begin{theorem}
  \label{thr:5}
  Let $\M=\M(\Sigma,I)$ be a trace monoid, with $\mu_\M(X)$ the
  single-variable M\"obius polynomial. We assume that $|\Sigma|>1$.
  Then:
\begin{enumerate}
\item\label{item:5} The following formal identity holds in\/
  $\ZZ[[X]]$:
\[
G_\M(X)=1/\mu_\M(X) \:. 
\]
In particular, $G_\M(X)$~is a rational series.
\item\label{item:6} The polynomial $\mu_\M(X)$ has a unique root of
  smallest modulus, say~$p_0$\,. This root is real and lies in
  $(0,1)$.  The radius of convergence of the power series $G_\M(z)$ is
  thus equal to~$p_0$\,, and the series $G_\M(z)$ is divergent
  at~$p_0$\,.
\item\label{item:7} If $N$ is the multiplicity of the root $p_0$
  in~$\mu_\M(X)$\,, then the following estimate holds for some
  constant $C>0$:
  \begin{equation}
    \label{eq:17}
    \lambda_\M(k)\sim_{k\to\infty} Ck^{N-1}\bigl(1/p_0 \bigr)^k\,.
  \end{equation}
  Furthermore, the root $p_0$ has multiplicity $1$ if\/ $\M$ is
  irreducible.
\end{enumerate}
\end{theorem}

\begin{lemma}
\label{lem:5}
  Let $\M$ be an irreducible trace monoid. Then for any
  non-empty clique $c\in\Cstar_\M$, we have\,:
\begin{equation}
  \label{eq:18}
  \lim_{k\to\infty} \lambda_{\M_c}(k) / \lambda_\M(k) =0\,.
\end{equation}
\end{lemma}

\begin{proof}
  The argument is rather standard. We sketch it and illustrate it
  for~$\M_1$\,.  Start by considering the Cartier-Foata automaton of
  $\M$ and transform it by expanding each node corresponding to a
  clique $c$ of cardinality strictly larger than one, into $|c|$
  nodes. The first of the expanded nodes is initial and the last of
  the expanded nodes is final. The non-expanded nodes are both initial
  and final.  See \cite[p.148]{krob03} for the details of the
  construction and see Figure~\ref{fig:illustration} for an
  illustration.

  Let $\A$ be the resulting automaton. Let $\A_c$ be the automaton
  obtained from $\A$ by keeping the same nodes, the same initial and
  final nodes, but by keeping only the arcs entering into the nodes
  labeled by the letters of~$\Sigma_c$\,.  Admissible paths of length
  $k$ in $\A$ are in bijection with traces of length $k$ in~$\M$.  And
  admissible paths of length $k$ in $\A_c$ are in bijection with
  traces of length $k$ in~$\M_c$ (recall that a path in an automaton
  is \emph{admissible} if it starts with an initial state and ends up
  with a final state).  Denote by $A$ the incidence matrix of the
  automaton~$\A$, and by $A_c$ the one of~$\A_c$.  By construction, we
  have:
\begin{align}\label{eq:pf}
A_c&\leq A,& A_c&\neq A\,.
\end{align}
According to Lemma~\ref{lem:1} below, since $\M$ is irreducible, the
matrix $A$ is primitive. So we are in the domain of applicability of
\cite[Theorem~1.2~p.9 and Theorem~1.1 point~$(e)$~p.4]{seneta81}, the
strong version of Perron-Frobenius Theorem for non-negative
matrices. According to it, the strict inequality~\eqref{eq:pf} yields
that the spectral radius of $A_c$ is strictly smaller than the
spectral radius of~$A$. The limit~\eqref{eq:18} follows.
\end{proof}

In Figure~\ref{fig:illustration}, we illustrate the construction of
the proof for the trace monoid $\M_1=\langle a,b,c\;|\;a\cdot b=b\cdot
a\rangle$ by showing the automaton~$\A$, to be compared with the
original automaton depicted in
Figure~\ref{fig:carteirfoatacliques}. In $\A$, the initial nodes are
$\{a,b,c,(a\cdot b)_1\}$ and the final nodes are $\{a,b,c,(a\cdot
b)_2\}$. For the clique~$\{a\}$, for instance, the automaton
$\A_{\{a\}}$ has the same nodes as $\A$ but a single arc: the
self-loop around the node~$b$.
\begin{figure}[ht]
\centerline{\SelectTips{cm}{}
\xymatrix@C=3em@R=1.5em{
  &{(a\cdot b)_1}\ar@/^/[rr] &&{(a\cdot
    b)_2}\ar@/^/[dr]\ar@/^/[dlll]\ar@/_/[ddl]!U\ar@/^/[ll]\POS!R(.2)!U(.5)&
  {}\save[]+<4em,0em>*\txt<10em>{Two states resulting from the
    expansion of the original state $a\cdot b$}\ar@{-->}[l]\restore
  \save"1,2"."1,4"!R="P"*+[F--]\frm{}\restore
  \\
  {a}\ar@/^/[drr]\POS!U(.2)!L(.5)\ar@(ul,dl)[]!D(.2)!L(.5)&&&&
  {b}\ar@/_/[dll]\POS!D(.2)!R(.5)\ar@(dr,ur)[]!U(.2)!R(.5)
  \\
  &&{c}\ar@/^/[ull]\ar@/_/[urr]!D\ar@/_/[uul]\POS!L(.2)!D(.5)\ar@(dl,dr)[]!R(.2)!D(.5)
}}\par\bigskip\medskip
  \caption{\textsl{The expanded automaton $\A$. }}
\label{fig:illustration}
\end{figure}
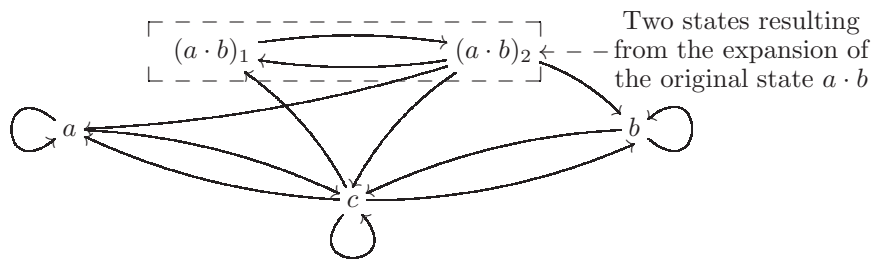

\begin{proposition}
  \label{cor:5}
  Let $\M=\M(\Sigma,I)$ be an irreducible trace monoid, and let $p_0$
  be the root of smallest modulus of the M\"obius
  polynomial~$\mu_\M(X)$\,. Then $\muc(p_0)>0$ for all non-empty
  cliques $c\in\Cstar_\M$\,.
\end{proposition}

\begin{proof}
  By point~\ref{item:7} of Theorem~\ref{thr:5}, and since $\M$ is
  assumed to be irreducible, we have the estimate $\lambda_\M(k)\sim
  C(1/p_0)^k$ for $k\to\infty$. Let $c\in\Cstar_\M$\,. If $c$ is
  maximal, then $\muc(X)=1$ and thus $\muc(p_0)>0$ holds
  trivially. Otherwise, the monoid $\M(\Sigma_c,I_c)$ is non-trivial,
  let $p_c$ be the root of smallest modulus of~$\muc$\,. Then
  $\lambda_{\M_c}(k)\sim C'\,k^{N-1}(1/p_c)^k$ for some constant $C'>0$ and
  where $N$ is the multiplicity of $p_c$ in~$\muc(X)$\,, by
  point~\ref{item:7} of Theorem~\ref{thr:5}.  In
  view of~Lemma \ref{lem:5}, it follows that $p_c>p_0$\,.

  In particular, and by point~\ref{item:6} of Theorem~\ref{thr:5},
  $p_0$~is in the open disc of convergence of the series
  $G_{\M_c}(z)=\sum_{k\geq0}\lambda_{\M_c}(k)z^k$\,. Applying
  point~\ref{item:5} of Theorem~\ref{thr:5} to the trace
  monoid~$\M_c$\,, and converting the formal equality into an equality
  between reals, we get $ \muc(p_0)=1/{G_{\M_c}(p_0)}>0$\,, completing the
  proof.
\end{proof}

\subsection{M\"obius inversion formula and consequences}
\label{sec:mobi-invers-form}

The notion of M\"obius function of a partial order is due to
Rota~\cite{rota64}. Due to the formal equality $G_\M(X)\mu_\M(X)=1$,
recalled in Theorem~\ref{thr:5} point~\ref{item:5}, the M\"obius
function, element of the incidence algebra in the sense of Rota, is
easily found to be $\nu_\M:\M\times\M\to\ZZ$ given by
$\nu_\M(x,y)=(-1)^{|y|-|x|}$ if $x\leq y$ and if $y-x$ is a clique, and
$0$ otherwise. From this, we derive the following form of the M\"obius
inversion formula \cite[Prop.2]{rota64} for trace monoids.

\begin{proposition}
  \label{thr:4}
Let $\C$ be the set of cliques associated with an independence pair $(\Sigma,I)$.
 If $f,h:\C\to\RR$ are two functions, then $h$~is the M\"obius
transform of $f$, that is, satisfies:
\begin{equation}
\label{eq:23}
\forall c\in\C \qquad h(c)=\sum_{c'\in\C\tq c\leq c'}(-1)^{|c'|-|c|}f(c')\:,
\end{equation}
if and only if the following holds:
\begin{equation}
  \label{eq:14}
  \forall c\in\C\qquad f(c)=\sum_{c'\in\C\tq c'\geq c}h(c')\,.
\end{equation}
\end{proposition}

The formula~\eqref{eq:14} is called the \emph{M\"obius inversion
  formula}, since it allows to recover any function $f:\C\to\RR$ from
its M\"obius transform. We give an enhanced version in
Proposition~\ref{thr:3}  below which applies outside the mere set~$\C$.

\begin{corollary}
  \label{cor:3}
Let $\C$ be the set of cliques associated with an independence pair $(\Sigma,I)$.
  Let $h:\C\to\RR$ be the M\"obius transform of a function
  $f:\C\to\RR$ such that $f(0)=1$. Then $h(0)=0$ if and only if\/
  $\sum_{c\in\Cstar}h(c)=1$.
\end{corollary}

\begin{proof}
  The M\"obius inversion formula~\eqref{eq:14} applied to $c=0$ writes
  as follows: $1=h(0)+\sum_{c\in\Cstar}h(c)$, whence the result. 
\end{proof}

We recall that, by convention, the M\"obius transform of a valuation
$f:\M\to\RR_+^*$ is defined as the M\"obius transform of its
restriction to~$\C_\M$\,. 

\begin{proposition}
  \label{prop:4}
  Let $h:\C\to\RR$ be the M\"obius transform of a valuation
  $f:\M\to\RR_+^*$\,, where $\C$ is associated to an independence pair
  $(\Sigma,I)$. Assume that $h(0)=0$, and let $g:\C\to\RR$ be the
  function defined by:
  \begin{equation*}
    \forall c\in\C\quad g(c)=\sum_{c'\in\C\tq c\to c'}h(c')\,.
  \end{equation*}
Then the following formula holds:
\begin{equation*}
  \forall c\in\C\quad g(c)f(c)=h(c).
\end{equation*}
\end{proposition}

\begin{proof}
  The identity $g(0)f(0)=h(0)$ is trivial since $0\to c'$ if and only
  if $c'=0$, and thus $g(0)=h(0)$, while $f(0)=1$. For $c\in\Cstar$ a
  non-empty clique, and by definition of $h$ and of~$g$, one has:
  \begin{align*}
    g(c)=\sum_{c'\in\Cstar}(-1)^{|c'|}f(c')\sum_{\substack{\delta\in\Cstar\tq
\delta\leq c'\wedge c\to\delta
}}\1{c\to\delta}\1{\delta\leq
    c'}(-1)^{|\delta|}\,.
  \end{align*}

For any $c'\in\Cstar$, the range of $\delta$ in the above sum is
$\bigl\{\delta\in\Cstar\tq \delta\leq c'\cap\{\alpha\in\Sigma\tq
c\to\alpha\}\bigr\}$, and the  binomial formula yields thus:
\begin{equation*}
  \sum_{\substack{\delta\in\Cstar\tq
\delta\leq c'\wedge c\to\delta
}}(-1)^{|\delta|}=-\tau_c(c')\,,\quad\text{with }\tau_c(c')=
  \begin{cases}
    0,&\text{if $c'\indep c$},\\
1,&\text{if $\neg(c'\indep c)$}.
  \end{cases}
\end{equation*}
We obtain thus:
\begin{equation}
  \label{eq:30}
  g(c)=-\sum_{c'\in\Cstar}(-1)^{|c'|}f(c')\tau_c(c')\,.
\end{equation}

The assumption $h(0)=0$ writes as:
\begin{equation}
  \label{eq:31}
  1+\sum_{c'\in\Cstar\tq c'\indep c}(-1)^{|c'|}f(c')+\sum_{c'\in\Cstar\tq\tau_c(c')=1}(-1)^{|c'|}f(c')=0\,.
\end{equation}
Combining~\eqref{eq:30} and~\eqref{eq:31} yields:
\begin{equation}
  \label{eq:32}
  g(c)=1+\sum_{c'\in\Cstar\tq\ c'\indep c}(-1)^{|c'|}f(c')\,.
\end{equation}
We multiply both sides of~\eqref{eq:32} by $f(c)$ and apply the change
of variable $c''=c\cdot c'$.  Using that $f$ is multiplicative, this
yields:
\begin{equation*}
  f(c)g(c)=f(c)+\sum_{c''\in\Cstar\tq c''>c}(-1)^{|c''|-|c|}f(c'')=h(c)\,,
\end{equation*}
which was to be proved.
\end{proof}

Next result is a generalization of the M\"obius inversion
formula~\eqref{eq:14}. Whereas the original M\"obius inversion formula
is valid for any function $f:\C_\M\to\RR$\,, the generalized version
applies to valuations only. 

Let $f:\M\to\RR_+^*$ be a valuation. In \eqref{eq:23}, the M\"obius
transform of $f$ was defined as a function $h:\C_\M\to\RR$. Here, we
extend the domain of definition of $h$ to the whole monoid $\M$ as
follows.  If $u\in\M$ is a non-empty trace, we write $u=v\cdot c$\,,
where $c\in\Cstar_\M$ is the \emph{last} clique in the Cartier-Foata
normal form of~$u$, and $v$ is the unique trace such that $u=v\cdot c$
holds.  The {\em extended M\"obius transform} $h:\M\to\RR$ is then
defined by:
\begin{equation}
\label{eq:42}
  \forall u\in\M\qquad h(u)=f(v)h(c)\,.
\end{equation}

\begin{proposition}
  \label{thr:3}
  Let $f:\M\to\RR$ be a valuation defined on a trace monoid.  Let
  $u\in\M$ be a non-empty trace, and let $\M(u)$ denote the set:
  \begin{equation}
    \label{eq:37}
    \M(u)=\{u'\in\M\tq \height(u')=\height(u)\,,\;u\leq u'\}\,,
  \end{equation}
  where $\height(\cdot)$ is the height function defined
  in\/~\S~{\normalfont\ref{sec:order-trac-extend}}.  Then we have the identity:
\begin{equation}
  \label{eq:34}
\sum_{u'\in\M(u)}h(u')=f(u)\,,
\end{equation}
where $h:\M\to\RR$ is the extended M\"obius transform of $f$ defined
in\/~\eqref{eq:42}.
\end{proposition}

\begin{proof}
  We fix a non-empty trace $u\in\M$, and we set
  \begin{equation*}
    S_0=\sum_{u'\in\M(u)}h(u')\,.
  \end{equation*}
  Let $c_1\to\ldots\to c_n$ be the Cartier-Foata normal form of~$u$,
  and set $c=c_n$ and $v=c_1\cdot\ldots\cdot c_{n-1}$. We apply Lemma~\ref{prop:1} to derive:
\begin{align}
\label{eq:43}
S_0
&=\sum_{(\gamma_1,\ldots,\gamma_n)\in J(c_1,\ldots,c_n)}
  h(c_1\cdot\gamma_1\cdot\ldots\cdot c_n\cdot \gamma_n)\,,
\end{align}
where we have set:
\begin{equation*}
\begin{split}
  J(c_1,\ldots,c_n)=\bigl\{(\gamma_1,\ldots,\gamma_n)\in\C^n\tq
\gamma_i\indep c_i,\ldots,c_n\text{ for $1\leq i\leq n$\,,}\\
c_1\cdot\gamma_1\to\ldots\to c_n\cdot\gamma_n\bigr\}\,.
\end{split}
\end{equation*}
By definition of~$h$, this yields:
\begin{align}
  S_0&=f(v)S_1\,,&\text{with\quad}S_1&=\sum_{(\gamma_1,\ldots,\gamma_n)\in
    J(c_1,\ldots,c_n)}f(\gamma_1)\dots f(\gamma_{n-1})h(c_n\cdot\gamma_n)\,.
\end{align}
We define, for $x,y\in\C$:
\begin{equation}
\label{eq:44}
\lambda(x,y)=\sum_{\delta\in\C\tq (x\to\delta)\wedge (\delta\geq y)}h(\delta)\,.
\end{equation}
Rewriting $S_1$ using the above notation, we get:
\begin{align*}
  S_1&=\sum_{(\gamma_1,\ldots,\gamma_{n-1})\in
    K(c_1,\ldots,c_n)} f(\gamma_1)\cdots f(\gamma_{n-1})\lambda(c_{n-1}\cdot\gamma_{n-1},c_n)\,,
\end{align*}
where we have set:
\begin{multline*}
  K(c_1,\ldots,c_{n})=\bigl\{(\gamma_1,\ldots,\gamma_{n-1})\in\C^{n-1}\tq
\gamma_i\indep c_i,\ldots,c_n\text{ for $1\leq i\leq n-1$,}\\
c_1\cdot\gamma_1\to\ldots\to c_{n-1}\cdot\gamma_{n-1}\bigr\}\,.
\end{multline*}
Applying Lemma~\ref{lem:3} below yields, for any $\gamma_{n-1}$ in the
scope of the sum defining~$S_1$\,:
\begin{align*}
\lambda(c_{n-1}\cdot\gamma_{n-1},c_n)&=f(c_n)\sum_{\delta\in\C\tq\delta\indep c_{n-1}\cdot \gamma_{n-1},\,c_n}(-1)^{|\delta|}f(\delta)\,.
\end{align*}
Therefore $S_1=f(c_n)S_2$ where $S_2$ is defined by:
\begin{equation}
\label{eq:45}
  S_2=\sum_{\substack{(\gamma_1,\ldots,\gamma_{n-1})\in
      K(c_1,\ldots,c_n)\\
\delta\in\C\tq \delta\indep c_{n-1}\cdot\gamma_{n-1},\,c_n
}}
f(\gamma_1)\cdots f(\gamma_{n-1})(-1)^{|\delta|}f(\delta)\,.
\end{equation}
Since $S_2=1$ according to Lemma~\ref{lem:4} below, we conclude
that $S_1=f(c_n)$ and finally $S_0=f(v)f(c_n)=f(u)$, which was to be proved.
\end{proof}

\medskip
In the course of the above proof, we have used the following two lemmas.

\begin{lemma}
  \label{lem:3}
If $x,y\in\C$ are two cliques such that $x\to y$ holds, then the
quantity $\lambda(x,y)$ defined in~\eqref{eq:44} satisfies:
\begin{equation*}
  \lambda(x,y)=f(y)\sum_{\delta\in\C\tq\delta\indep x,y}(-1)^{|\delta|}f(\delta)\,.
\end{equation*}
\end{lemma}

\begin{proof}
    By definition of the M\"obius transform~$h$, one has:
  \begin{align}
\label{eq:25}    \lambda(x,y)& \ =\sum_{\delta\in\C\tq (x\to\delta)\wedge
  (y \leq \delta)}h(\delta) \ =\sum_{z\in\C\tq y\leq
      z}(-1)^{|z|}f(z)H(x,y,z)\\
\notag &\qquad \text{where } \quad H(x,y,z)=\sum_{\delta\in\C\tq (x\to\delta)
    \wedge (y\leq
  \delta\leq z)}(-1)^{|\delta|}\,.
  \end{align}

Consider $\delta$ as in the sum defining $H(x,y,z)$. Since $x\to y$
holds by assumption, the following equivalence holds: $x\to\delta\iff
x\to(\delta-y)$. The binomial formula yields thus:
\begin{equation*}
  H(x,y,z)=\begin{cases}
(-1)^{|y|},&\text{if $(z-y)\indep x$\,,}\\
0,&\text{otherwise.}
  \end{cases}
\end{equation*}

Reporting the latter value of $H(x,y,z)$ in~\eqref{eq:25} and
considering the change of variable $z=y\cdot\delta$ yields the
expected expression for~$\lambda(x,y)$.
\end{proof}

\begin{lemma}
  \label{lem:4}
For any integer $n\geq1$ and for any cliques $c_1,\ldots,c_n$ such
that $c_1\to\ldots\to c_n$ holds, the quantity $S_2$ defined
in~\eqref{eq:45} satisfies~$S_2=1$.
\end{lemma}

\begin{proof}
  We substitute the variable $\delta'=\delta\cdot\gamma_{n-1}$ to
  $\delta$ in the defining sum for~$S_2$. For each $\gamma_{n-1}$ in
  the scope of the sum, one has $\gamma_{n-1}\indep c_{n-1},c_n$\,, as
  specified by the definition of $K(c_1,\ldots,c_n)$. Hence the set
  $\{\delta\in\C\tq \delta\indep c_{n-1}\cdot\gamma_{n-1},\,c_n\}$
  corresponds to the set $\{\delta'\in\C\tq
  (\delta'\geq\gamma_{n-1})\wedge (\delta'\indep c_{n-1}\,,\,c_n) \}$,
  and the change of variable yields:
\begin{gather}
\label{eq:53}
  S_2=\sum_{\substack{\delta\in\C\tq
\delta\indep c_{n-1},c_n}}
(-1)^{|\delta|}f(\delta)
\sum_{ (\gamma_1,\ldots,\gamma_{n-2}) \in L(c_1,\ldots, c_n)}  f(\gamma_1)\cdots f(\gamma_{n-2})
R(\gamma_{n-2})\,,\\
\notag
\begin{split}
\text{with\quad}  L(c_1,\ldots,c_n)=\bigl\{(\gamma_1,\ldots,\gamma_{n-2})\in\C^{n-2}\tq
\gamma_i\indep c_i,\ldots,c_n\text{ for $1\leq i\leq n-2$,}\\
  c_1\cdot\gamma_1\to\ldots\to c_{n-2}\cdot\gamma_{n-2}\bigr\}
\end{split}\\
\notag
\text{and\quad}R(\gamma_{n-2})=\sum_{\substack{\gamma_{n-1}\in\C\tq\\
\gamma_{n-1}\indep c_{n-1}\,,\,c_n\\
\gamma_{n-1}\leq\delta\\
c_{n-2}\cdot\gamma_{n-2}\to\gamma_{n-1}
}}(-1)^{|\gamma_{n-1}|}\,.
\end{gather}

In the scope of the sum defining $R(\gamma_{n-2})$, the
condition ``$c_{n-2}\cdot \gamma_{n-2}\to c_{n-1}\cdot\gamma_{n-1}$'' has been
replaced by ``$c_{n-2}\cdot\gamma_{n-2}\to\gamma_{n-1}$'', which is
equivalent since $c_{n-2}\to c_{n-1}$ already holds by assumption.

Since $\delta\indep c_{n-1},c_n$\,, and by the binomial formula,
the sum defining $R(\gamma_{n-2})$ evaluates as follows:
\begin{equation}
  \label{eq:52}
  R(\gamma_{n-2})=\1{\delta\indep c_{n-2}\cdot\gamma_{n-2}}\,.
\end{equation}
Substituting the right side of~\eqref{eq:52} into~\eqref{eq:53}, we
obtain:
\begin{align*}
  S_2&=\sum_{\begin{subarray}{l}
(\gamma_1,\ldots,\gamma_{n-2})\in L(c_1,\ldots, c_n) \\
\delta\in\C \tq \delta \indep c_{n-2}\cdot\gamma_{n-2},c_{n-1},c_n
\end{subarray}
}
f(\gamma_1)\cdots f(\gamma_{n-2})(-1)^{|\delta|}f(\delta)\,.
\end{align*}
Applying recursively the same transformation eventually yields:
\begin{align*}
  S_2&=\sum_{\begin{subarray}{l}
\quad\gamma,\delta\in\C\tq\\
\gamma\indep c_1,\ldots,c_n\\
\phantom{\gamma}\llap{\scriptsize$\delta$}\indep c_1\cdot\gamma,c_2,\ldots,c_n
\end{subarray}
}
f(\gamma)(-1)^{|\delta|}f(\delta)\,,
\end{align*}
and after yet the same transformation:
\begin{align*}
  S_2&=\sum_{\delta\in\C\tq\delta\indep
    c_1,\ldots,c_n}(-1)^{|\delta|}f(\delta)\sum_{\gamma\leq\delta} (-1)^{|\gamma|}=f(0)=1\,,
\end{align*}
completing the proof.
\end{proof}

\subsection{Combinatorial lemmas}
\label{sec:combinatorial-lemmas}

The following result is known, see for
instance~\cite[Lemma~3.2]{krob03}. We provide an alternative proof
below.

\begin{lemma}
  \label{lem:1}
  If $\M$ is an irreducible trace monoid, then $(\Cstar_\M,\to)$ is a
  strongly connected graph.
\end{lemma}

\begin{proof}
  Consider the following claim~$(*)$, which we prove under the
  hypothesis that $(\Sigma,D)$ is connected:
\begin{center}
  \begin{minipage}{.10\linewidth}
    $(*)$%
  \end{minipage}%
  \begin{minipage}[c]{.90\linewidth}
\noindent
Let $c$ be a non-empty clique of $(\Sigma,I)$, and let $\alpha_0\in\Sigma$
be a letter such that $\alpha_0\indep c$ holds. Then there exists an
integer $p\geq 1$ and $p$ non-empty cliques $\gamma_1,\ldots,\gamma_p$
such that $\gamma_p=c\cdot\alpha_0$ and $c\to\gamma_1\to\ldots\to\gamma_p$
holds.
  \end{minipage}
\end{center}

Indeed, since $c\neq0$, pick $\alpha_1\in c$. Since $(\Sigma,D)$ is
assumed to be connected, there is a sequence of letters
$\beta_1,\ldots,\beta_p\in\Sigma$ such that, putting
$\beta_0=\alpha_1$, one has $(\beta_i,\beta_{i+1})\in D$ for all
$i\in\{0,\ldots,p-1\}$, and $\beta_p=\alpha_0$\,. Next, for each letter
$\alpha\in c$, consider the following integer:
\begin{equation*}
  i(\alpha)=\min\bigl\{j\in\{1,\ldots,p\}\tq
\alpha\indep\beta_j,\beta_{j+1},\ldots,\beta_p\bigr\}\,.
\end{equation*}
Since $\beta_p=\alpha_0$, and since $\alpha_0\indep c$ by assumption, one has
indeed $i(\alpha)\leq p$ for all $\alpha\in c$. Consider the sequence
of cliques $\gamma_1,\ldots,\gamma_p$ defined as follows:
\begin{equation*}
  \forall j\in\{1,\ldots,p\},\qquad
\gamma_j=\{\beta_j\}\cup\{\alpha\in c\tq j\geq i(\alpha)\}.
\end{equation*}

We leave it to the reader to check that $\gamma_1,\ldots,\gamma_p$
thus defined satisfy the claim~$(*)$. The statement of the lemma
follows easily from the claim.
\end{proof}

Next lemma will be a key in proving the uniqueness of uniform
measures.  We recall first that for any trace monoid $\M=\M(\Sigma,I)$,
the \emph{mirror} mapping $\rev:\M\to\M$ is defined as the quotient
mapping of the mapping $\Sigma^*\to\Sigma^*$ defined on words by
$\rev(\alpha_1\cdots\alpha_n)=\alpha_n\cdots\alpha_1$\,. Given
$u\in\M$, the heap of $\rev(u)$ is obtained from the heap of $u$ by
considering it upside-down. If $s_1\to\ldots\to s_k$ and
$r_1\to\ldots\to r_\ell$ are the respective Cartier-Foata decompositions of $u$
and $\rev(u)$, then:
\begin{align}
\label{eq-upside}
k&=\ell\,,&&
\begin{cases}
  r_k\leq s_1\\
r_{k-1}\cdot r_k\leq s_2\cdot s_1\\
\cdots\\
r_1\cdot r_2\cdot\ldots\cdot r_k\leq s_k\cdot s_{k-1}\cdot\ldots\cdot s_1
\end{cases}
%
\end{align}
The properties in \eqref{eq-upside} are easy to visualize using the heap
interpretation.

\begin{lemma}[Hat lemma]
  \label{lem:7}
\textbf{(Hat lemma)}\quad
Let $\M$ be an irreducible trace monoid. Then there exists a trace
$w\in\M$ with the following property:
\begin{equation}
  \label{eq:29}
  \forall u,v\in\M\qquad (|u|=|v|)\wedge (u\neq v)\ \Longrightarrow \ \up(u\cdot
  w)\;\cap\up(v\cdot w)=\emptyset\,.
\end{equation}
\end{lemma}

\begin{proof}
  Let $\M=\M(\Sigma,I)$, and let $D$ be the associated dependence
  relation.  Since $\M$ is assumed to be irreducible, we consider a
  sequence $(\alpha_i)_{1\leq i\leq q}$ with $\alpha_i\in\Sigma$ such
  that:
  \begin{inparaenum}[1)]
    \item every $\alpha\in\Sigma$ occurs at least once in the
      sequence; and
    \item $(\alpha_i,\alpha_{i+1})\in D$ for all $i\in\{1,\ldots,q-1\}$.
  \end{inparaenum}
We introduce the trace
\begin{equation*}
w=\alpha_1\cdot\alpha_2\cdot
\ldots\cdot\alpha_{q-1}\cdot\alpha_q\cdot\alpha_{q-1}\cdot\alpha_{q-2}\cdot
\ldots\cdot\alpha_1\,,
\end{equation*}
and we aim at showing that $w$ satisfies~\eqref{eq:29}.
\begin{itemize}
\item[\itshape Claim $(*)$] For all $u\in\M$, the \emph{first} $q$
  cliques in the Cartier-Foata normal form of $w\cdot u$ are
  $\alpha_1\to\alpha_2\to\ldots\to\alpha_q$\,, and the \emph{last} $q$
  cliques in the Cartier-Foata normal form of $u\cdot w$ are
  $\alpha_q\to\alpha_{q-1}\to\ldots\to\alpha_1$\,.
\end{itemize}

We prove the claim~$(*)$. By construction, the Cartier-Foata
decomposition of $w$ is
\begin{equation*}
  \alpha_1\to\alpha_2\to\ldots\to\alpha_{q-1}\to\alpha_q\to\alpha_{q-1}\to\ldots\to\alpha_2\to\alpha_1\,.
\end{equation*}

Consider the trace $w\cdot u$ for some $u\in\M$. Let $d_1\to\ldots\to
d_p$ be the Cartier-Foata decomposition of~$w\cdot u$\,.  Applying
Lemma~\ref{prop:1} to traces $w$ and $w\cdot u$, we conclude in
particular that $2q-1\leq p$, and for all $i\in\{1,\ldots,q\}$, we
have $d_i=\alpha_i\cdot\gamma_i$ for some clique $\gamma_i\in\C$ such
that
$\gamma_i\indep\alpha_i,\ldots,\alpha_q,\alpha_{q-1},\ldots,\alpha_2,\alpha_1$\,. Since
$\alpha_q,\ldots,\alpha_1$ range over all letters of~$\Sigma$, it
follows that $\gamma_i=0$. So we have proved that $d_1\to\ldots\to
d_q=\alpha_1\to\ldots\to\alpha_q$\,. 

Now for the second part of the claim~$(*)$, consider the trace $u\cdot
w$ for some $u\in\M$. We have $\rev(u\cdot
w)=\rev(w)\cdot\rev(u)=w\cdot\rev(u)$. According to the above, the
Cartier-Foata decomposition of $w\cdot\rev(u)$ starts with
$\alpha_1\to\ldots\to\alpha_q$\,. According to~\eqref{eq-upside}, the
$q$ last cliques $d_1\to\ldots\to d_q$ of $u\cdot w=\rev\bigl(w\cdot
\rev(u)\bigr)$ satisfy:
\begin{align*}
  d_q&\leq \alpha_1\,,&
d_{q-1}\cdot
d_q&\leq\alpha_2\cdot\alpha_1\,,&&\cdots&
d_1\cdot\ldots\cdot d_q&\leq \alpha_q\cdot\ldots\cdot\alpha_1\,.
\end{align*}
Since the $\alpha_i$ are minimal in~$\Cstar$, it follows that
$d_q=\alpha_1\,,\ d_{q-1}=\alpha_2\,,\ldots,d_1=\alpha_q$\,, which
completes the proof of the claim~$(*)$. 

We now come to the proof of~\eqref{eq:29} for~$w$. Let $u,v\in\M$ such
that $|u|=|v|$ and $\up (u\cdot w)\;\cap\up( v\cdot
w)\neq\emptyset$. According to \eqref{eq:38}, it follows that $u\cdot
w$ and $v\cdot w$ are compatible. Hence there are $u',v'\in\M$ such
that $u\cdot w\cdot u'=v\cdot w\cdot v'$\,. Set
$\wbar=\alpha_1\cdot\alpha_2\cdot\ldots\cdot\alpha_{q-1}$\,, so that
$w=\wbar\cdot\alpha_q\cdot\rev(\wbar)$. It follows from the
claim~$(*)$ that the Cartier-Foata decomposition of
$u\cdot\wbar\cdot\alpha_q$ is of the form $c_1\to\ldots\to
c_k\to\alpha_q$\,, and the Cartier-Foata decomposition of
$\alpha_q\cdot\rev(\wbar)\cdot u'$ is of the form $\alpha_q\to
d_1\to\ldots\to d_\ell$\,. Hence the Cartier-Foata decomposition of
$u\cdot w\cdot u'$ is obtained by concatenating the ones of $u\cdot w$
and of~$u'$\,. By the same argument, the Cartier-Foata decomposition
of $v\cdot w\cdot v'$ is obtained by concatenating the ones of $v\cdot
w$ and of~$v'$\,.

Hence, by uniqueness of the Cartier-Foata
decomposition of $u\cdot w\cdot u'=v\cdot w\cdot v'$\,, between the
decompositions of $u'$ and of~$v'$\,, one is a suffix of the other. On
the other hand, $u\cdot w\cdot u'=v\cdot w\cdot v'$ and $|u|=|v|$
imply $|u'|=|v'|$, and therefore $u'=v'$. Since $\M$ is cancellative,
we conclude that $u=v$, which completes the proof.
\end{proof}

\section{Proofs of the main results}
\label{part:proofs-results}

\subsection{From Bernoulli
  measures to Markov chains and \goodname\ valuations}
\label{sec:necess-cond-from}

In this section, we prove Proposition~\ref{prop:9} and
point~\ref{item:8} of Theorem~\ref{thr:2} and point~\ref{item:3} of
Theorem~\ref{thr:1}. The three results correspond to \emph{necessary}
conditions for a probability measure on the boundary of a trace monoid
to be Bernoulli. We start with the two latter points.

The setting is the following: we consider an irreducible trace
monoid $\M=\M(\Sigma,I)$, and we assume that $\pr$ is a Bernoulli
measure defined on $(\B\M,\FFF)$. We consider the valuation
$f:\M\to\RR_+^*$ associated with~$\pr$, defined by $f(u)=\pr(\up u)$
for $u\in\M$, and we let $h:\C\to\RR$ be the M\"obius transform
of~$f$.

We start by proving formula~\eqref{eq:13} in Theorem~\ref{thr:2},
which implies most of the other affirmations. Hence, let $n\geq1$ be
an integer and let $c_1\to\ldots\to c_n$ be $n$ non-empty
cliques. According to formula~\eqref{eq:9} in
Proposition~\ref{prop:3}, one has:
\begin{equation}
\label{eq:20}
\pr(C_1=c_1,\ldots,C_n=c_n)=f(u)-\delta\,,\quad\text{with
}\delta=\pr\Bigl(\bigcup_{\substack{c\in\C\,:\\c>c_n}}\up(v\cdot c)\Bigr)\,,
\end{equation}
where $v=c_1\cdot\ldots\cdot c_{n-1}$ and $u=c_1\cdot\ldots\cdot
c_n$\,. For any $\xi\in\B\M$, one has $\xi\in\up(v\cdot c)$ for some
clique $c>c_n$ if and only if there is a letter $\alpha\indep c_n$
such that $\xi\in\up(v\cdot c_n\cdot\alpha)$. Let
$\{\alpha_1,\ldots,\alpha_r\}$ be an enumeration of those letters
$\alpha\indep c_n$\,. Applying the inclusion-exclusion
principle, we obtain:
\begin{align*}
  \delta&=\sum_{k=1}^r(-1)^{k+1}\sum_{1\leq i_1<\cdots <i_k\leq
    r}\pr\bigl( \up(v\cdot c_n\cdot
  \alpha_{i_1})\cap\cdots\cap\up(v\cdot c_n\cdot\alpha_{i_k})\bigr)\,.
\end{align*}

For $i_1,\ldots,i_k$ indices as in the above sum, 
put
$\gamma=\{\alpha_{i_1},\ldots,\alpha_{i_k}\}$ and $\gamma'= c_n\cdot \gamma$\,. The related
intersection is then 
either empty if $\gamma$ is not a clique, or
equal to $\up(v\cdot c_n\cdot \gamma)= \up(v\cdot \gamma')$.
  if $\gamma$ is a clique, which
 is equivalent to $\gamma'=c_n\cdot\gamma$ being a clique. 
By construction, the cliques $\gamma'$ range over the cliques $c'\in\C$
such that $c'> c_n$\,, and thus, taking into account that $\pr(\up\cdot\,)=f(\,\cdot\,)$ is
multiplicative, we obtain:
\begin{align*}
  \delta&=\sum_{c'\in\C\tq c'>c_n} (-1)^{|c'|-|c|+1}\pr\bigl(\up(v\cdot
  c')\bigr) = f(v) \sum_{c'\in\C\tq c'>c_n} (-1)^{|c'|-|c|+1} f(c')\,.
\end{align*}

Injecting the above  in~\eqref{eq:20}, and
writing $f(u)=f(v)f(c_n)$, we get:
\begin{equation*}
  \pr(C_1=c_1,\ldots,C_n=c_n)=
f(v)\sum_{c'\in\C\tq c'\geq c_n} (-1)^{|c'|-|c|} f(c') 
= f(v)h(c_n)\,,
\end{equation*}
Since $f(v)h(c_n)=f(c_1)\cdots f(c_{n-1})h(c_n)$, we have the desired
result. 

As a particular case for $n=1$, it follows at once that $h\rest\Cstar$
coincides with the probability distribution of~$C_1$\,. Therefore, by
the total probability law, $\sum_{c\in\Cstar}h(c)=1$, and by
Corollary~\ref{cor:3}, $h(0)=0$. It remains only to prove that $h>0$
on $\Cstar$ to obtain that $f$ is a \goodname\ valuation.

For this, let $c\in\Cstar$, and let $c'$ be a maximal clique
in~$\Cstar$. Since $(\Cstar,\to)$ is connected according to
Lemma~\ref{lem:1}, there exists a sequence $c_1,\ldots,c_n$ of cliques such
that $c_1=c$, $c_n=c'$, and $c_i\to c_{i+1}$ holds for all
$i\in\{1,\ldots,n-1\}$. Since $c'$ is maximal, the definition of the
M\"obius transform yields $h(c')=f(c')$, and thus
$\pr(C_1=c_1,\ldots,C_n=c_n)=f(c_1)\cdots f(c_n)>0$. This implies that
$h(c)=\pr(C_1=c)>0$. We have proved that $f$ is a \goodname\
valuation, and completed the proof of point~\ref{item:3} in
Theorem~\ref{thr:1}.

Finally, it remains only to show that $(C_n)_{n\geq1}$ is an aperiodic
and irreducible Markov chain with the specified transition matrix,
since the law of $C_1$ has already been identified
as~$h\rest\Cstar$\,.

From the general formula  proved above, we
derive, if $c_1\to\ldots\to c_n$ holds:
\begin{align*}
  \pr(C_1=c_1,\ldots,C_n=c_n|C_1=c_1,\ldots,C_{n-1}=c_{n-1})&=\frac1{h(c_{n-1})}f(c_{n-1})h(c_n)\,.
\end{align*}

Since $h(0)=0$, it follows from Proposition~\ref{prop:4}, and using
the same notation~$g$, that
$h(c_{n-1})=f(c_{n-1})g(c_{n-1})$\,. Therefore:
\begin{align*}
  \pr(C_1=c_1,\ldots,C_n=c_n|C_1=c_1,\ldots,C_{n-1}=c_{n-1})&=\frac{h(c_n)}{g(c_{n-1})}\,.
\end{align*}

Since the latter quantity only depends on $(c_{n-1},c_n)$, it follows
that $(C_n)_{n\geq1}$ is a Markov chain with the transition matrix
described in the statement of point~\ref{item:8} of
Theorem~\ref{thr:2}. 

The chain is irreducible since $(\Cstar,\to)$ is connected, as already
observed. And it is aperiodic since $c\to c$ holds for any
$c\in\Cstar$. The proof is complete.

\paragraph{Proof of Proposition~\ref{prop:9}.}
\label{sec:proof-proposition-1}

Consider any measure of finite mass on~$\B\M$. Let $u\in\M$ be any
trace. The elementary cylinder $\up u$ writes as the union: $\up
u=\bigcup_{\alpha\in\Sigma}\up(u\cdot\alpha)$. Applying
inclusion-exclusion principle as above yields the expected
formula~\eqref{eq:26}.

\subsection{From \goodname\ valuations to Bernoulli measures, through
  Markov chains}
\label{sec:suff-cond-from}

In this section, we consider a trace monoid $\M$ equipped with a
\goodname\ valuation $f:\M\to\RR_+^*$\,, and we establish the
existence and uniqueness of a probability measure on $(\B\M,\FFF)$
such that $f(\,\cdot\,)=\pr(\up\,\cdot\,)$. This corresponds to the
proof of point~\ref{item:4} of Theorem~\ref{thr:1} and of
point~\ref{item:9} of Theorem~\ref{thr:2}.

It must be noted that we do not use the irreducibility of~$\M$ in this
part of the proof.

\medskip

The uniqueness of $\pr$ follows from the remark made in
\S~\ref{sec:finite-meas-bound} that elementary cylinders form a
$\pi$-system generating~$\FFF$.

For proving the existence of~$\pr$, we proceed by considering first
the Markov chain on the Cartier-Foata subshift which is necessarily
induced by~$\pr$, if it exists (even though it was only established
for an irreducible trace monoid). Let $h:\C\to\RR$ be the M\"obius
transform of the \goodname\ valuation~$f$. By assumption, $h(0)=0$,
and therefore, thanks to Corollary~\ref{cor:3},
$\sum_{c\in\Cstar}h(c)=1$. Since $h>0$ on~$\Cstar$ by assumption, it
follows that $h\rest\Cstar$ defines a probability distribution
on~$\Cstar$.

Furthermore, the normalization factor defined by
\begin{equation*}
  g(c)=\sum_{c'\in\Cstar\tq c\to c'}h(c')
\end{equation*}
is non-zero on~$\Cstar$. Hence the stochastic matrix
$P=(P_{c,c'})_{(c,c')\in\Cstar\times\Cstar}$ is well defined by
\begin{equation*}
  P_{c,c'}=
\begin{cases}
h(c')/g(c),&\text{if $c\to c'$} \\
  0,&\text{if $\neg(c\to c')$}
\end{cases} \:.
\end{equation*}

Let $\prq$ be the probability measure on the space
$(\Omega,\GGG)$, corresponding to the law of the
Markov chain on $\Cstar$ with $h\rest\Cstar$ as initial measure
and with $P$ as transition matrix. Let finally $\pr$ be the
probability measure on $(\B\M,\FFF)$ associated with~$\prq$. Then we
claim that $\pr(\up u)=f(u)$ holds for all traces $u\in\M$.

First, we observe that, for any integer $n\geq1$ and any sequence of
cliques $\delta_1\to\dots\to\delta_n$\,, the following identity holds:
\begin{equation}
  \label{eq:46}
  \pr(C_1=\delta_1,\ldots,C_n=\delta_n)=f(\delta_1)\cdots f(\delta_{n-1})h(\delta_n)\,.
\end{equation}

Indeed, $h(0)=0$ by assumption, and this implies $h=fg$ on $\C$
according to Proposition~\ref{prop:4}. Using the form of the
transition matrix $P$ and the definition of the initial law of the
chain $(C_n)_{n\geq1}$\,, we have thus:
\begin{equation*}
  \pr(C_1=\delta_1,\ldots,C_n=\delta_n)=
h(\delta_1)\frac{h(\delta_2)}{g(\delta_1)}\dots\frac{h(\delta_n)}{g(\delta_n)}=
f(\delta_1)\cdots f(\delta_{n-1})h(\delta_n)\,,
\end{equation*}
which proves~\eqref{eq:46}.  We recognize the generalized form of the
M\"obius transform introduced in~\eqref{eq:42} for the valuation~$f$,
and obtain thus:
\begin{equation}
  \label{eq:47}
  \pr(C_1=\delta_1,\ldots,C_n=\delta_n)=h(\delta_1\cdot\ldots\cdot\delta_n)\,.
\end{equation}

We now prove $\pr(\up u)=f(u)$ for $u\in\M$.  Trivially, $\pr(\up
0)=f(0)=1$. Let $u$ be a non-empty trace, and let $n=\height(u)$ be
the height of~$u$. It follows from~\eqref{eq:8}
stated in Proposition~\ref{prop:3} that we have:
\begin{equation}
\label{eq:21}
 \pr(\up u)=\pr(C_1\cdot\ldots\cdot C_n\geq u)\,.
\end{equation}

The random trace $C_1\cdot\ldots\cdot C_n$ ranges over traces of
height~$n$. Combining~\eqref{eq:47} and~\eqref{eq:21} yields thus:
\begin{equation*}
  \pr(\up u)=\sum_{u'\in\M\tq \height(u')=\height(u),\;u'\geq u}h(u')\,.
\end{equation*}

By Proposition~\ref{thr:3}, we deduce that $\pr(\up u)=f(u)$, as claimed.
This completes the proofs of point~\ref{item:4} of Theorem~\ref{thr:1}
and of point~\ref{item:9} of Theorem~\ref{thr:2}.

\subsection{Uniform measures: existence and uniqueness}
\label{sec:unif-meas-exist}

This section is devoted to the proof of Theorem~\ref{thr:6} and of
Proposition~\ref{prop:7}.

We consider an irreducible trace monoid $\M=\M(\Sigma,I)$, and we let
$p_0$ be the unique root of smallest modulus of~$\mu_\M$\,, which is
well defined according to Theorem~\ref{thr:5}. Let also
$f_0(u)=p_0^{|u|}$ be the uniform valuation associated to~$p_0$\,.

\paragraph{Existence of a uniform Bernoulli measure.}
\label{sec:existence}

We aim at applying Theorem~\ref{thr:1} to obtain the existence of a
probability measure $\pr$ on $\B\M$ such that $\pr(\up
\cdot\,)=f_0(\,\cdot\,)$ on~$\M$\,.

Accordingly, we only have to check that the uniform valuation $f_0$ is
a \goodname\ valuation. As already noted
in~\S~\ref{sec:uniform-measures}, if $h:\C\to\RR$ is the M\"obius
transform of~$f_0$\,, the condition $h(0)=0$ is equivalent to $p_0$
being a root of~$\mu_\M$\,, which is fulfilled. According to the
equivalence stated in Definition~\ref{def:3}, the condition $h>0$ on
$\Cstar$ amounts to check that $\muc(p_0)>0$ for all $c\in\Cstar$, and
this derives from Proposition~\ref{cor:5}, since $\M$ is assumed to be
irreducible. Hence $f_0$ is indeed a \goodname\ valuation, which
implies the existence of the desired probability measure.

\paragraph{Uniqueness of the uniform measure.}
\label{sec:uniqueness}

The uniqueness of uniform probability measures entails the uniqueness
of Bernoulli uniform measures, hence we restrict ourselves to proving
the following: if $(\gamma_n)_{n\geq0}$ is a sequence of real numbers
such that
\begin{equation}
\label{eq:3}
\forall n\geq0\quad  \forall u\in\M\quad|u|=n\implies
\pr(\up u)=\gamma_n\,,
\end{equation}
then $\gamma_n=p_0^n$ for all $n\geq0$.

Let $\lambda_n=\lambda_\M(n)$ denote the number of traces of length $n$ in $\M$ for
all integer $n\geq 0$. Consider the following two generating series:
\begin{align}\label{eq-StildeS}
G(X) &= \sum_{n\geq0} \lambda_n X^n\,,& S(X) &= \sum_{n\geq0} \gamma_n X^n \,.
\end{align}

According to Theorem~\ref{thr:5} point~\ref{item:5}, we have
$G(X)=1/\mu_\M(X)$ where $\mu_{\M}(X)$ is the M\"obius polynomial
of~$\M$. By developing  $G(X) \mu_{\M}(X)$, we obtain in
particular:
\begin{equation}\label{eq-rec1}
\forall n \geq \max_{c\in\C} |c| \qquad \sum_{c\in \C} (-1)^{|c|}
\lambda_{n - |c|} = 0 \:.
\end{equation}

Now let us turn our attention to~$S(X)$. According to
Proposition~\ref{prop:9}, we have: $\sum_{c\in \C}(-1)^{|c|}
\pr\bigl(\up( u\cdot c)\bigr) =0$ for
all $u\in \M$. Using~\eqref{eq:3}, it translates as:
\begin{equation}
\label{eq-rec2}
\forall n \geq 0 \qquad \sum_{c\in \C} (-1)^{|c|}
\gamma_{n + |c|} = 0 \:.
\end{equation}
In view of~\eqref{eq-rec1} and~\eqref{eq-rec2}, we are steered to
consider $G(X)$ and $S(X)$ as being sort of dual.  We are going to
build upon this.

\medskip

Equation~\eqref{eq-rec2} can be rewritten as $\gamma_n= \sum_{c\in
  \Cstar}(-1)^{|c|+1} \gamma_{n+|c|}$. By injecting this identity
in~$S(X)$, we get
\begin{align*}
S(X)  & \ = \ \sum_{n\geq0} \ \Bigl( \  \sum_{c\in \Cstar}
(-1)^{|c|+1} \gamma_{n+|c|} \Bigr) \ X^n \\
 & \ = \ \sum_{c\in \Cstar} \ (-1)^{|c|+1} X^{-|c|} \Bigl( S(X) -
 \sum_{i=0}^{|c|-1} \gamma_i \ X^i \Bigr) \:.
\end{align*} 
Collecting the different terms involving~$S(X)$, we recognize the
coefficients of the M\"obius polynomial $\mu_\M(X)$ and obtain:
\begin{equation*}
  S(X)\mu_\M(1/X)=\sum_{c\in\Cstar}(-1)^{|c|}X^{-|c|}\,\Bigl(\;\sum_{i=0}^{|c|-1}\gamma_iX^i\;\Bigr)\,.
\end{equation*}
Note that this proves already that $S(X)$ is rational.

Set $\ell=\max_{c\in \C} |c|$. Then $\mu_\M(X)$ is a polynomial of
degree~$\ell$. Let $p_0,\ldots,p_{\ell-1}$ be the roots of~$\mu_\M(X)$\,,
with $p_0<|p_1|\leq|p_2|\leq\ldots\leq|p_{\ell-1}|$\,. Denoting by
``$\propto$'' the proportionality relation, we have 
$\mu_\M(1/X)\propto X^{-\ell}(1-p_0X)\cdots(1-p_{\ell-1}X)$\,, which yields:
\begin{gather}
  \label{eq:35}
  S(X)\propto\frac{P(X)}{(1-p_0X)\cdots(1-p_{\ell-1}X)}\propto\frac{P(X)}{(X-1/{p_0})\cdots
    (X-1/{p_{\ell-1}})}\,,\\
\notag
P(X)=\sum_{c\in\Cstar}(-1)^{|c|}X^{\ell-|c|}\sum_{i=0}^{|c|-1}\gamma_iX^i\,.
\end{gather}
We observe that $P$ is a polynomial of degree at most $\ell-1$.

Let $w$ be a trace as in the hat
lemma~\ref{lem:7}, that is, satisfying~\eqref{eq:29}. 
Set $|w|=q$. Define, for all integers $n \geq q$, the set $\D_n
= \{ u\cdot w \mid u \in \M_{n-q} \}$ where $\M_k=\{u\in\M\tq |u|=k\}$
for all integers $k\geq0$.  Observe that $\D_n \subseteq \M_n$ and, by
cancellativity of the trace monoid~$\M$, that $\D_n$~is in bijection
with~$\M_{n-q}$\,. Hence:
\begin{equation}
\label{eq:36}
|\D_n| = |\M_{n-q}| \sim_{n\to
  \infty} C_1\,(1/p_0)^n\:,
\end{equation}
for some constant $C_1>0$, according to Theorem~\ref{thr:5}
point~\ref{item:7}.  The cylinders $\up u$ for $u$ ranging over $\D_n$
are disjoint by construction of~$w$, we have thus \mbox{$\sum_{u \in {\D}_n}
\pr( \up u) \leq 1$}. But, according to~\eqref{eq:3}, we have $\sum_{u \in \D_n}
\pr( \up u) = |\D_n| \gamma_n$\,. So we get $|{\mathcal D}_n| \cdot
\gamma_n \leq 1$. Using~\eqref{eq:36}, we obtain
\begin{equation}
\label{eq:33}
 \forall n \geq 0 \qquad  \gamma_n \leq C_2\,
p_0^n\:,
\end{equation}
for some constant $C_2>0$.

Returning to the expression~\eqref{eq:35} for~$S$, the roots of the
denominator are:
$1/|p_{\ell-1}|\leq1/|p_{\ell-2}|\leq\cdots\leq1/|p_1|<1/p_0$\,. Hence,
would any of the roots $1/p_j$ with $j>0$ not be a root of the
numerator~$P$, that would prevent~\eqref{eq:33} to hold. Since $P$ is
of degree at most $\ell-1$, we deduce that $1/p_{\ell-1},\ldots,1/p_1$
are exactly all the roots of~$P$, and~\eqref{eq:35} rewrites as $S(X)=
K/(1-p_0X)$ for some constant $K\neq0$. Evaluating both members at
$X=0$ yields $K=1$ since $\gamma_0=1$, and thus
$S(X)=1/(1-p_0X)$\,. Since $S(X)=\sum_{n\geq0}\gamma_nX^n$ by
definition, we obtain that $\gamma_n=p_0^n$ for all $n\geq0$, and the
proof is complete.

\paragraph{Proof of Proposition~\ref{prop:7}.}

Let $p_0$ be the unique root of smallest modulus of the M\"obius
polynomial. Start with the valuation $f$ defined by
$f(\alpha)=p_0$ for all $\alpha$ in~$\Sigma$. By Theorem~\ref{thr:6},
$f$~is a \goodname\ valuation.

Now consider a collection of reals
$\varepsilon=(\varepsilon_{\alpha})_{\alpha \in \Sigma}$ such that
$p_0+\varepsilon_{\alpha} \in (0,1)$ for all $\alpha\in \Sigma$.  Let
$f_{\varepsilon}$ be the valuation defined by:
$f_{\varepsilon}(\alpha) = p_0+\varepsilon_{\alpha}$ for all
$\alpha\in \Sigma$. Let $h_{\varepsilon}$ be the associated M\"obius
transform. The goal is to show that there exist a continuous family of
values for $(\varepsilon_{\alpha})_{\alpha \in \Sigma}$ such that
$f_{\varepsilon}$ is \goodname, that is:
\begin{align*}
(i) \ h_\varepsilon(0) &=0\,,& (ii) \ \forall
c\in\Cstar\quad h_{\varepsilon}(c)&>0 \,.
\end{align*}

First, observe that condition $(ii)$ is an open condition and that it
is satisfied for $\varepsilon=0$\,. So it is still satisfied if
$|\varepsilon_{\alpha}|$ is small enough, for all $\alpha \in\Sigma$.

Now let us concentrate on~$(i)$. Fix a letter $a\in \Sigma$. The
equation $h_\varepsilon(0) =0$ is an affine equation in
$\varepsilon_a$ if the values $\varepsilon_\alpha$ for $\alpha\neq a$
are fixed:
\begin{equation*}
  \label{eq:27}
  (p_0+\varepsilon_a) A_\varepsilon+B_\varepsilon=0\,,
\end{equation*}
with 
\begin{equation*}
A_\varepsilon=\sum_{c\in \C \tq a\in c} (-1)^{|c|}
   \prod_{\alpha \in c, \alpha\neq a}(p_0+\varepsilon_{\alpha})\,,\quad
B_\varepsilon=\sum_{c\in\C\tq a\notin c}(-1)^{|c|}\prod_{\alpha \in c}(p_0+\varepsilon_{\alpha})\,.
\end{equation*}
Observe that we have:
\begin{equation*}
A_0=\sum_{c\in\C\tq a\in c}(-1)^{|c|}p_0^{|c|-1}\,,\quad B_0=\sum_{c\in\C\tq a\notin c}(-1)^{|c|}p_0^{|c|}\,.
\end{equation*}
We recognize in $B_0$ the M\"obius polynomial of the
independence pair $(\Sigma',I')$, with $\Sigma'=\Sigma\setminus\{a\}$
and $I'=I\cap(\Sigma'\times\Sigma')$, evaluated at~$p_0$\,. But $\M$
is irreducible, and as already observed in the proof of
Proposition~\ref{cor:5}, the comparison of growth rates entails that
$p_0$ is strictly smaller in modulus than all the roots of the
polynomial~$\mu_{\M(\Sigma',I')}$\,. Hence $B_0\neq0$. Since
$p_0A_0+B_0=0$\,, we conclude that $A_0\neq0$ and thus
$A_\varepsilon\neq0$ for $\varepsilon$ small enough. Consequently, the
equation $(p_0+\varepsilon_a) A_\varepsilon+B_\varepsilon=0$ has a unique
solution in $\varepsilon_a$ if all $\varepsilon_\alpha$ are small
enough for $\alpha\neq a$. Since $|\Sigma|>1$ by assumption, there is
indeed an uncountable number of values for
$(\varepsilon_\alpha)_{\alpha\neq a}$ arbitrarily close to~$0$.

We have proved the existence of a continuous family of distinct
\goodname\ valuations, and we conclude by Theorem~\ref{thr:1}.

\section{Conclusion and perspectives}
\label{sec:persp-future-work}

The paper has introduced and characterized Bernoulli measures on
irreducible trace monoids, interpreted as a probabilistic model of
concurrent systems with a memoryless property. The combinatorics of
trace monoids plays a central role in the characterization of
Bernoulli measures, and in particular the notion of M\"obius transform
and of M\"obius polynomial. The existence and uniqueness of uniform
measures has been established, as well as the fact that they belong to
the class of Bernoulli measures. A realization result allows for
effective sampling of Bernoulli measures, paving the way for future
applications.

The extension to non-\pirreducible\ trace monoid works out nicely and
the complete description of Bernoulli measures is postponed to a
future work.

Further developments of this work can be expected.  First, it is
natural to adapt our construction to trace groups.  It would also be
interesting to generalize our approach to other monoids or
groups. Braid monoids and groups, and more generally Artin monoids and
groups of finite Coxeter type, are natural candidates.

Another extension consists in studying Markovian measures on infinite
traces instead of Bernoulli measures.  Applications to the
construction of Markovian measures for the executions of $1$-bounded
Petri nets are expected. In this case, the executions are described as
a regular trace language.

\printbibliography

\end{document}